\documentclass[english]{amsart}

\usepackage{babel}
\usepackage{amstext}
\usepackage{amsmath}
\usepackage{amsfonts}
\usepackage{latexsym}
\usepackage{ifthen}
\usepackage{xypic}
\xyoption{all}
\newcommand{\FS}{{\mathbb F}}
\newcommand{\F}{{\mathcal F}}

\newcommand{\PN}{{\mathbb P}}

\newcommand{\PF}{{\mathbb F}}
\newcommand{\rk}{{\rm rk}}

\newcommand{\Pic}{{\rm Pic}}

\newcommand{\lra}{\longrightarrow}
\newcommand{\KC}{{\mathbb C}}
\newcommand{\KZ}{{\mathbb Z}}
\newcommand{\KQ}{{\mathbb Q}}

\newcommand{\Bs}{{\rm Bs}}

\newcommand{\Bl}{{\rm Bl}}
\newcommand{\exc}{{\rm exc}}
\newcommand{\X}{{\mathcal X}}

\newcommand\sE{{\mathcal E}}
\newcommand\sF{{\mathcal F}}

\newcommand\sI{{\mathcal I}}

\newcommand\sO{{\mathcal O}}

\newcommand\bZ{{\mathbb Z}}

\newcommand\bQ{{\mathbb Q}}
\newcommand\bN{{\mathbb N}}

\newcommand\bP{{\mathbb P}}

\newcounter{lemma}

\newtheorem{lemma1}[lemma]{\setcounter{equation}{0}}

\newenvironment{lemma}{\begin{lemma1}{\bf Lemma.}}{\end{lemma1}}

\newenvironment{theorem}{\begin{lemma1}{\bf Theorem.}}{\end{lemma1}}
\newenvironment{theorem2}[1]{\begin{lemma1}{\bf Theorem [#1].}}{\end{lemma1}}
\newenvironment{proposition}{\begin{lemma1}{\bf Proposition.}}{\end{lemma1}}
\newenvironment{proposition2}[1]{\begin{lemma1}{\bf Proposition [#1].}}{\end{lemma1}}
\newenvironment{corollary}{\begin{lemma1}{\bf Corollary.}}{\end{lemma1}}

\newenvironment{remark}{\begin{lemma1}{\bf Remark.}\rm}{\end{lemma1}}

\newenvironment{notation}{\begin{lemma1}{\bf Notation.}}{\end{lemma1}}

\newenvironment{setup}{\begin{lemma1}{\bf Setup.}}{\end{lemma1}}

\title {Threefolds with big and nef anticanonical bundles II}
\date{Version vom 15.10.07}

\author[P. Jahnke]{Priska Jahnke} 
\address{P. Jahnke - 
Mathematisches Institut - Universit\"at Bayreuth - D-95440 Bayreuth, Germany} 
\email{priska.jahnke@uni-bayreuth.de}
\author[Th. Peternell]{Thomas Peternell}
\address{Th. Peternell - 
Mathematisches Institut - Universit\"at Bayreuth - D-95440 Bayreuth, Germany} 
\email{thomas.peternell@uni-bayreuth.de}
\author[I.Radloff]{Ivo Radloff} 
\address{I.Radloff - 
Mathematisches Institut - Universit\"at Bayreuth - D-95440 Bayreuth, Germany} 
\email{ivo.radloff@uni-bayreuth.de}

\begin{document}

\maketitle
\tableofcontents

\section{Introduction}
This is the second part of our classification of smooth complex projective threefolds $X$ whose anticanonical bundles $-K_X$ are big and nef, but not ample. 
In other words, we classify {\it almost Fano threefolds}. 
In the first part [JPR05] we classified those $X$ with Picard number $\rho (X) = 2$, whose anticanonical morphism contracts a divisor. To be more
specific, recall that some multiple $-mK_X$ is generated by global sections (usually $m = 1 $) inducing a morphism with connected fibers
$$ \psi: X \to X'$$ 
which we call the {\it anticanonical morphism} of $X$. Since $-K_X$ is big but not ample, $\psi$ is birational, but not an isomorphism. 
\vskip .2cm \noindent
In this paper we are concerned with the case $\rho(X) = 2$ and $\psi$ {\it small}, i.e., $\psi$ contracts finitely many curves and
nothing else. The singular variety $X'$ is Fano with terminal Gorenstein singularities, but it is not $\bQ-$factorial. By [Na97], $X'$ admits a smoothing $(X_t)$, 
and $\rho (X_t) = 1$ by [JR06a]. 
In this situation $X$ can be flopped, i.e., there is another almost Fano threefold $X^+$ with small anticanonical morphism $\psi^+: X^+ \to X'$
which henceforth is isomorphic to $X$ in codimension 1. This is an important tool to study the original threefold $X.$ Another main ingredient is the
unique Mori contraction $\phi: X \to Y$, which is somehow perpendicular to $\psi.$ Since $\rho(X^+) = 2$, too, also $X^+$ carries a unique
Mori contraction $\psi^+: X^+ \to Y^+$, and we study the interplay of the two contractions.
\vskip .2cm \noindent
Our classification results are resumed in the lists in the appendix. We classify according to the types of $\phi$ and $\phi^+$: del Pezzo fibrations over
 $\bP_1$ (including projective and quadric bundles), conic bundles over $\bP_2$, and birational contractions. However, due to the complexity of the 
problem and the length of the paper, we shall not consider the case that both contractions are birational, and hope to come back to that case later. 
\vskip .2cm \noindent In many cases we explicitly write down examples, but in some circumstances this is very delicate and must be left open. The reason for that
is twofold. First, it is difficult to construct explicity del Pezzo fibrations with relative Picard number $2$ and $K_F^2 = 5,6$ (with $F$ the general
fiber) and second, in the case of conic bundles $X \subset \bP(E)$, it is possible to write down the potential rank 3 bundles $E$ over $\bP_2,$ but in order
to construct $X$, it is necessary to work out ``smooth'' sections in a certain twist of $S^2(E)$ which is difficult, too. 
\vskip .2cm \noindent 
We would like to thank the DFG-Schwerpunkt ``Globale Methoden in der komplexen Geometrie'' and the DFG-Forschergruppe ``Classification of algebraic surfaces
and compact complex manifolds'' for the support of our project.

\section{Preliminaries}

\begin{notation} {\rm As in Part I we consider a smooth projective threefold $X$ with $-K_X$ big and nef.  We always assume that $X$ is {\it not Fano} 
and say that $X$ is {\it almost Fano}. 
Then $-mK_X$ will be spanned for suitable large $m.$
Throughout this paper $\psi: X \to X'$ will denote the
morphism (with connected fibers) associated with $\vert {-}mK_X \vert$ {\bf and $\psi$ is assumed to be small}, therefore contracts 
only finitely many smooth rational curves and nothing else. 
Notice that $X'$ is Gorenstein Fano threefold with only terminal singularities and $\rho(X') = 1$, but $X'$ is not
$\bQ-$factorial.}
\end{notation} 

By \cite{Kollar}, there exists the following flop--diagram 
\begin{equation} \label{flopdiag}
\xymatrix{X \ar@{-->}[rr]^{\chi} \ar[dr]^{\psi} \ar[d]_{\phi} & & X^+
\ar[dl]_{\psi^+} \ar[d]^{\phi^+}\\
Y & X' & \tilde{Y}}
\end{equation}
where the rational map $\chi$ is an isomorphism outside the
exceptional locus of $\psi$ and $X^+$ is again a smooth almost
Fano threefold with anticanonical map $\psi^+$ and extremal
contraction $\phi^+$. Note that $\phi$ and $\phi^+$ are not
necessarily of the same type. Our assumption $\rho(X) = 2$ implies that $\chi$ does
not depend on the choice of some $\psi$--negative divisor in $X$. To be more precise, we have the following

\begin{proposition} Let $D$ be any divisor which is not $\psi-$nef, i.e. $-D$ is $\psi-$ample. 
Then the $D-$flop of $\psi$ exists,
i.e. a small birational map $\psi^+: X^+ \to X'$ such that the strict transform $\tilde D \subset X^+$ is $\psi^+-$ample. Moreover $X^+$ is smooth with $-K_{X^+}$ big and nef and
$$\rho(X^+) = 2,$$
$$({-}K_X)^3 = ({-}K_{X^+})^3, $$ 
$$h^0(\sO_X(D)) = h^0(\sO_{X^+}(\tilde D)),$$
the same being true for the strict transform of any divisor on $X.$ 
Finally $\psi^+$ does not depend on $D.$
\end{proposition}

\proof Let $l$ be a curve contracted by $\psi.$ Then $D \cdot l < 0,$
hence the $D-$flop exists by \cite{Kollar}. Also the smoothness and the
statement on the Picard number follows from \cite{Kollar}. Since $\psi^+$ is
small and since $X'$ has only terminal singularities, we have $K_{X^+} = (\psi^+)^*(K_{X'}), $ hence $-K_{X^+}$ is big and nef and also $({-}K_X)^3 = ({-}K_{X^+})^3.$ The 
$H^0-$statement is clear, too. Finally $\psi^+$ does not depend on
$D$, since $\rho(X) = 2$ so that two divisors $D$ and $D'$  which are negative on the curves $l_{\psi}$ coincide up to multiples in a neighborhood of the exceptional locus of $\psi.$  
\qed

\begin{notation} {\rm  The $D-$flop as in \eqref{flopdiag} will always denoted $\psi^+: X^+ \to X';$ if $E$ is a divisor on $X,$ then the strict transform 
of $E$ will be denoted by $\tilde E.$ \\
On the level of sheaves, let $L$ be the pull back to $X$ of the ample generator on $Y;$ then we set
$$ \tilde L = (\psi^+)^*(\psi_*(L)^{**}). $$
Notice that $-\tilde K_X = - K_{X^+}$. 
The induced birational map $X \dasharrow X^+$ is called $\chi.$ 
Since $\rho(X^+) = 2$ and since $-K_{X^+}$ is big and nef but not ample,
$X^+$ carries a unique contraction which is called  $\phi^+: X^+ \to
Y^+.$ Then as above we consider the pull back $L^+$ of an ample generator on $Y^+$ and
define $\tilde L^+$ a line bundle on $X$.}
\end{notation}

\

A {\em smoothing} of a singular Fano threefold $X'$ is a flat family
 \[\X \lra \Delta\]
over the unit disc, such that $\X_0 \simeq X'$ and $\X_t$ is a smooth
Fano threefold for $t \not= 0$. Namikawa has shown in \cite{Namikawa}
that a smoothing always exists if $X'$ has only terminal Gorenstein
singularities, not necessarily $\KQ$--factorial: In this case the
Picard groups of $X'$ and the general $\X_t$ are isomorphic (over
$\KZ$) by \cite{smoothing}. 

\begin{proposition2}{\cite{Namikawa}, \cite{smoothing}} \label{antmodel}
Let $X'$ be a Gorenstein Fano threefold with only
  terminal singularities (not necessarily $\bQ-$factorial). Then $X' $
  has a smoothing $\X \to \Delta$ and $\Pic(X') \simeq \Pic(\X_t)$. In particular, $X'$ and $\X_t$ have the
same Picard number and the same index.
\end{proposition2}

\begin{proposition}\label{basepoints}
The anticanonical bundle $-K_{X'}$ and therefore $-K_X$ are generated by global sections unless $X'$ is a deformation of the Fano threefold
$V_2$ and arises as complete intersection of a quadric cone and a
general sextic in the weighted projective space $\bP(1^4,2,3).$ In
this case, there exists a small resolution $X \to X'$ with $\rho(X) =
2$, $({-}K_X)^3 = 2$ and $X \simeq X^+$ both admit a del Pezzo
fibration with $K_F^2 = 1$.
\end{proposition} 

\begin{proof}
Assume $X'$ is a Gorenstein, not $\KQ$-factorial Fano threefold with terminal
singularities, such that $|{-}K_{X'}|$ is not base point free. By
\cite{JR}, $X' \subset \PN(1^4, 2,3)$ is a complete intersection of a
quadric cone $Q$ in the first four
variables and a general sextic. This means $X'$ is a double cover of
the cone over the quadric $Q \simeq \PN_1 \times \PN_1 \hookrightarrow
\PN_8$ embedded by the system $|(2,2)|$, i.e. we have
 \[\xymatrix{V \ar[rr]^{\hspace{-1cm}2:1} \ar[d]^q & & \PN(\sO \oplus \sO(2,2))
   \ar[d]^p \ar[r]^{\pi}
   & Q \simeq \PN_1 \times \PN_1\\
             X' \ar[rr]^{2:1} & & \widehat{Q}&}\] 
Here $V$ is a smooth almost Fano threefold with $\rho(V) = 3$, the
anticanonical divisor $-K_V$ being the pull back of
$\pi^*\sO(1,1)$. The double cover $V \to \PN(\sO \oplus \sO(2,2))$ is
ramified along the minimal section 
 \[Q_0\simeq \PN_1 \times \PN_1\] 
of the projective bundle and a general cubic, disjoint from
$Q_0$. The vertical maps $p$ and $q$ contract $Q_0$ and its reduced
inverse image $E$ in $V$ to a point. We have 
 \[K_V = q^*K_{X'} + E\]
and $E|_E = \sO(-1,-1)$ by adjunction formula.

\vspace{0.2cm}

Let $\zeta$ be the tautological line bundle on $\PN(\sO \oplus
\sO(2,2))$ and $F_1, F_2 \simeq \Sigma_2$ the pull-back of the two rulings of $Q$. The
contraction $p$ of $Q_0$ to a point factors over the blowdown of $Q_0$
to $\PN_1$ along the two rulings, defined by $|\zeta + F_i|$:
 \[\xymatrix{\PN(\sO \oplus \sO(2,2)) \ar[d]^{p_i} \ar[r]^{\pi}
   & Q \simeq \PN_1 \times \PN_1\ar[d]\\
    Z_i \ar[d] \ar[r] & \PN_1\\
   \widehat{Q} &}\] 
By construction, the maps $p_i$ are crepant, hence $Z_1$ and $Z_2$ are Gorenstein almost Fano threefolds with
canonical singularities along the image of $Q_0$. Let
 \[V \stackrel{q_i}{\lra} X_i \stackrel{\psi_i}{\lra} X'\]
be the induced factorization of $q: V \to X'$, i.e., $X_i$ is a double
cover of $Z_i$. Then $q_i$ contracts $E$ along a ruling to $\PN_1$,
but here $K_V$ is negativ on the fibers, hence $X_1$ and $X_2$ are smooth
almost Fano threefolds with $\rho(X_i) = 2$. The anticanonical map
$\psi_i: X_i \to X'$ is small with exceptional locus a single curve,
namely $q_i(E) = \mathrm{Bs}|{-}K_{X_i}| \simeq \PN_1$.

\vspace{0.2cm}
 
On the fibers $F_i \simeq \Sigma_2$, the map $p_i$ contracts the
minimal section, i.e., $Z_i \to \PN_1$ has general fiber the quadric
cone. The induced covering gives a smooth del Pezzo surface of degree
$1$. The flop diagram hence is
 \[\xymatrix{X_1 \ar[d] \ar[dr]^{\psi_1} \ar@{<-->}[rr]^{\chi} & & X_2
 \ar[d] \ar[dl]_{\psi_2}\\
     \PN_1 & X' & \PN_1}\]
with $X = X_1 \simeq X_2 = X^+$, but $\chi$ not an isomorphism.
\end{proof}

\vskip .2cm \noindent

\noindent {\bf From now on we shall assume for the rest of the paper that $-K_X$ is spanned.}

\vskip .2cm \noindent 

\begin{notation} {\rm As in Part I, we let $\mu : X' \to W$ be the finite part of the map associated with $\vert {-}K_X \vert.$ 
We know (see e.g. Part I) that either
$\mu $ is an isomorphism or that $\mu$ has degree $2$, in which case $X'$ is hyperelliptic. Furthermore as usual we 
let $r_X = r_{X'}$ be the index of $X'$ and define the genus $g$
of $X$ or $X'$ by 
$$ 2g-2 = {{({-}K_X)^3}\over {2}}. $$ }
\end{notation}

\vspace{0.2cm}

By Bertini's classifiaction of varieties of minimal degree
(\cite{Bertini}), $W$ is either $\PN_3$, a quadric, the Veronese cone
or a scroll, i.e. the image $\overline{\PF(d_1, d_2, d_3)}$ of a projective bundle 
\[\PF(d_1, d_2, d_3) = \PN(\sO_{\PN_1}(d_1) \oplus \sO_{\PN_1}(d_2) \oplus \sO_{\PN_1}(d_3)), \quad d_1 \ge d_2 \ge d_3 \ge 0\]
in $\PN_{d_1 + d_2 + d_3 + 2}$ under the map associated to the
tautological system $|\zeta|$. Denote the pencil of $\PF(d_1, d_2, d_3)$ by $|F|$. We obtain

\begin{proposition} \label{hyp}
Let $X$ be a smooth almost Fano threefold with $\rho(X) = 2$, such
that $\psi$ is small. If $X' \stackrel{2:1}{\lra} W$ is hyperelliptic,
then we are in one of the following cases 

\begin{enumerate} 
\item $(-K_X)^3 = 2$, $W = \PN_3$ and $X' \to W$ is ramified along a sextic; 
\item $(-K_X)^3 = 4$, $W \subset \PN_4$ is a quadric and $X' \to W$
is ramified along a qartic;
\item $(-K_X)^3 = 6$, $W \subset \PN_5$ is the singular scroll
  $\overline{\PF(2,1,0)}$ and either $X$ or $X^+$ is a double cover of
  $\PF(2,1,0)$, ramified along a general divisor in $|4\zeta - 2F|$;
\item $(-K_X)^3 = 8$, $W$ is the cone in $\PN_6$ over the Veronese surface in $\PN_5$ and $X' = X_6 \subset \PN(1^3,2,3)$, i.e., $X' \to W$ 
is ramified along a cubic.
\end{enumerate}
\end{proposition}

\noindent All of these threefolds except (3) are the expected
deformations of Iskovskikh's list. For (3) note, that in this case the
exceptional locus of $\psi$ consists of a single smooth rational
curve, which is contained in the ramification divisor, and contracted to a point by the map $\PF(2,1,0) \to W$. Moreover this case can be described explicitly: here $X$ (or $X^+$)
  admits a del Pezzo fibration with general fiber of degree $4$, and
  a smoothing $\X_t$ of $X'$ in the sense of Namikawa is a complete intersection of a cubic and a
  quadric in $\PN_5$. For further details and a construction of this threefold see \cite{smoothing}.

\vspace{0.2cm}

\noindent {\bf From now on we may assume that the only hyperelliptic cases are
(1), (2) and (4).}   

\begin{proof}
If $X'$ is hyperelliptic, then the image $W$ of $X'$ in $\PN_{g+1}$ is
a variety of minimal degree of Picard number one. By Iskovskikh's
classification it remains to consider 
 \begin{equation} \label{fact}
   \xymatrix{X \ar[rr]^{\psi} && X' \ar[d]^{\mu}\\
            \PF(d_1, d_2, d_3) \ar[rr]^{\sigma} && W}
 \end{equation}
for some $0 \le d_3 \le d_2 \le d_1$, i.e. $W$ is a (singular) scroll. Then $\rho(X') = 1$ implies $d_3 = 0$. If $d_2 = 0$, then $W$ is a double cone over a rational normal
curve of degree $d_1$. The double cover $X'$ will have canonical
singularities along a curve, which is impossible if $\psi$ is small.
Therefore $d_2 > 0$, i.e. $W$ is a cone over a Hirzebruch surface. 

Let now $\hat{F}$ in $X$ be the strict transform of the Weil divisor
$\sigma(F)$. Then $\hat{F}$ is Cartier and $\hat{F} \cdot l_{\psi} \not= 0$ for any curve
contracted by $\psi$. So after possibly replacing $X$ by its flop $X^+$ we may
assume 
 \begin{equation} \label{pos}
  \hat{F} \cdot l_{\psi} > 0.
\end{equation} 
Then $|\hat{F}|$ is a pencil by \cite{JR},
Lemma~6.1 or \cite{Ch99}. We shortly recall the argument. Assume that two general members $F_1, F_2 \in
|\hat{F}|$ are not disjoint and let $C \subset F_1 \cap F_2$ be any
irreducible curve. Then $C \subset F_1$ is a component of the
restriction of $F_2$ to $F_1$, which is contained in the exceptional locus of
$\psi$, hence contracted to points. This means $F_2 \cdot C \le 0$,
contradicting $\hat{F} \cdot C > 0$ by \eqref{pos}. 

This shows the system $|{-}K_X + m\hat{F}|$ defines a factorization of $X
\to W$ over the scroll, i.e. we get a map
 \[\nu \colon X \lra \PF(d_1, d_2, d_3)\]
completing \eqref{fact} into a commutative diagram. Considering the
Stein factorization, $\rho(X) = 2$ implies $\nu$ is a double
cover. We have $-K_X = \nu^*\zeta$ and the exceptional locus of
$\psi$ is mapped to the exceptional curve $C_0$ of $\sigma$. 
The ramification divisor is an element 
\[D \in |4\zeta - 2(d_1+d_2-2)F|.\]
We find $C_0 \subset D$. Moreover, for $(d_1, d_2)
\not= (1,1)$, $(2,1)$, $D$ will always be singular along $C_0$. 
\end{proof}

\vspace{0.2cm}

Although we have a description of $X'$ in case it is hyperelliptic, the precise
structure of $X$ itself is still not clear. The following proposition
can be found in \cite{AG5}, Remark~4.1.10:

\begin{lemma} \label{flopiso}
 Assume $X'$ is hyperelliptic. Denote the birational involution
 induced on $X$ by $\sigma$. If $W$ is $\KQ$--factorial, $\sigma$
 coincides with the flop on $X$. In particular, $X \simeq X^+$ as
  abstract varieties.
\end{lemma}

\begin{proof}
  Let $D$ be some divisor on $X$. Denote the strict transform under
  $\sigma$ by $D^{\sigma}$. Then $D + D^{\sigma}$ is the pull back of
  some $\sigma$--invariant (Weil-) divisor $B'$ on $X'$. Then $B'$ comes from
  $W$. As $W$ is $\KQ$--factorial, $mB'$ is Cartier. Then
    \[(D + D^{\sigma})\cdot l_{\psi} = \frac{1}{m} \psi^*(mB')\cdot l_{\psi} =
    0\]
  for any curve $l_\psi$ contracted by $\psi$. But then $D\cdot l_{\psi} =
    -D^{\sigma}\cdot l_{\psi}$. This implies $\sigma: X \dasharrow X$ is the flop.
\end{proof}

\begin{remark}
 The same is true when $-K_X = 2H$ with $H$ spanned defining some
 double cover of some $\KQ$--factorial $W$.
\end{remark}

The following corollary can be found in \cite{AG5}, Remark~4.1.10:

\begin{corollary} \label{hypsymm}
Assume $X'$ is hyperelliptic. Then $X \simeq X^+$ as abstract
varieties, except $X$ is a resolution of Proposition~\ref{hyp}, (3),
or $W \subset \PN_4$ is the quadric cone. In the latter case $({-}K_X)^3 = 4$
   and $X$ as well as $X^+$ is a double cover of $\PF(0,1,1)$, ramified along a divisor
   from $|4\zeta|$; they admit a del Pezzo fibration with $K_F^2 = 2$.
\end{corollary}

\begin{proof}
Let $X' \lra W$ be the double cover defined by $|{-}K_{X'}|$. If $W$ is
$\KQ$--factorial, the claim is just Lemma~\ref{flopiso} above. Case
(3) in Proposition~\ref{hyp} is explicitely described in
\cite{smoothing}. The only remaining case is the
quadric cone. But then analogously to the proof of
Proposition~\ref{hyp} either $X$ or the
flop $X^+$ is a double cover of the small resolution $\PF(0,1,1)$ of
the quadric cone.

Assume $X \to \PF(0,1,1)$ is that double cover. Since the quadric cone
admits two (isomorphic) small resolutions connected by a flop, we get
 
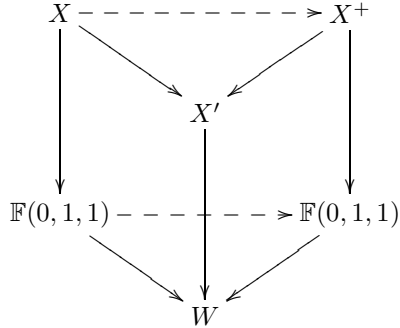
\[\xymatrix{X \ar[dd] \ar@{-->}[rr] \ar[dr] && X^+ \ar[dd] \ar[dl]\\
             & X' \ar[dd]  &\\
             \PF(0,1,1) \ar[dr]\ar@{-->}[rr] && \PF(0,1,1) \ar[dl]\\
              & W &}\]
meaning $\mu$ lifts to $X^+$ as well. The induced map $X \to \PN_1$ is a del Pezzo fibration, where the
general fiber $F$ is a double cover of $\PN_2$, ramified along the
restriction of the ramification divisor of $X \to W_0$, which gives a
quartic. Hence $K_F^2 = 2$.
\end{proof}

For small genus we find in our situation:

\begin{proposition} \label{ci}
Let $X$ be a smooth almost Fano threefold with $\rho(X) = 2$, such that the anticanonical map $\psi: X \to X'$ is small. Assume $X'$ not hyperelliptic.
 \begin{enumerate}
\item If $g = 3$, then $X'_4 \subset \PN_4$ is a quartic.
\item If $g = 4$, then $X'_{2,3} \subset \PN_5$ is a complete intersection of a quadric and a cubic.
\item If $g = 5$, then $X'_{2,2,2} \subset \PN_6$ is either a
complete intersection of three quadrics or $X \subset
\PF(1,1,1,0)$ is a divisor in $|3\zeta -F|$. In the latter case $X'$ is trigonal.
\end{enumerate}
\end{proposition}

\begin{proof}
Since the canonical curve section $C \subset X'$ is a smooth canonical curve of genus $g$, (1) and (2) are easily obtained. Assume $g = 5$. 
We have two possible cases: either $X'$ is cut out by
quadrics or it is trigonal. Since $X'$ is already a complete
intersection in the first case, assume the latter one. Then by
\cite{Cheltsov}, $X'$ is the anticanonical model of an almost Fano
threefold $V$ with canonical singularities, where $V$ is a divisor in
$|3\zeta -F|$ on one of $\PF(1,1,1,0)$ or
$\PF(2,1,0,0)$. The latter case is impossible, since here $X'$ is
singular along a curve.
\end{proof}

\

Assume now that $-K_X$ is divisible in $\Pic(X)$, i.e.,
$-K_X = r_XH$ for some $H \in \Pic(X)$ and $r_X \ge 2$ the
index. By assumption, then $H$ is big and nef, hence $|mH|$ is base
point free for all $m \gg 0$. Since $|mH|$ and $|(m+1)H|$ define the
same map for $m \gg 0$, we find 
 \[H = \psi^*H'\]
for some $H' \in \Pic(X')$, and hence $-K_{X'} = r_XH'$. By \cite{Shin} then $r_X \le
4$, with equality only for $X' = \PN_3$, and $r_X = 3$ implies $X'
\subset \PN_4$ is a quadric. We obtain:

\begin{proposition} \label{highindex}
If $r_X \ge 3$, then $X' \subset \PN_4$ is the cone over a smooth
quadric $Q \simeq \PN_1 \times \PN_1 \subset \PN_3$, and $X =
\PN(\sO_{\PN_1} \oplus \sO_{\PN_1}(1)^{\oplus 2})$ is the
small resolution of the vertex. In particular, $X \simeq X^+$.
\end{proposition}

The case $r_X = 2$ was treated in a more general situation in \cite{delpezzo};
we obtain the following list for $\rho(X) = 2$:

\begin{theorem2}{\cite{delpezzo}} \label{index2}
Assume $\rho(X) = 2$, $\psi$ is small and $r_X = 2$. Then $\phi: X \to Y$
is either a quadric bundle, or a $\PN_1$-bundle, or birational. 
\begin{enumerate}
\item If $\phi$ is a quadric bundle, then $X$ belongs to the following list. 
 \begin{enumerate}
   \item $X \subset \PN_3 \times \PN_1$ from $|(2, 2)|$, here $d =
      2$, $X^+ \simeq X$ and $\X_t \to \PN_3$ is a double cover, 
   \item $X \subset \FS(0^3, 1)$ from $|2\zeta+F|$, here $d = 3$,
   $X^+ = \Bl_p(V_{2,4})$ and $\X_t \simeq V_{2,3}$ (this is case (3), (iii)),
   \item $X \subset \FS(0^2, 1^2)$ from $|2\zeta|$, here $d = 4$,
     $X^+$ is of the same type and $\X_t \simeq V_{2,4}$,
   \item $X \subset \FS(0, 1^3)$ from $|2\zeta-F|$, here $d = 5$, $X^+ = \PN(\F)$ with some stable rank two bundle $\F \in
     {\mathcal M}(-1,2)$ (this is case (2), (i)), and
     $\X_t \simeq V_{2,5}$,
  \end{enumerate}
\item If $\phi$ is a $\PN_1$--bundle, then $X = \PN(\F)$ with a stable rank $2$ bundle on
$\PN_2$ with $c_1(\F) = -1$ and $2 \le c_2(\F) \le 5$. Moreover,
$\F(2)$ is nef, but not ample and has only finitely many jumping
lines. We have
 \begin{enumerate}
   \item $c_2(\F) = 2$. Then $d = 5$, $X^+$ admits a del Pezzo
     fibration as in (1), (iv) and $\X_t \simeq V_{2,5}$,
   \item $c_2(\F) = 3$. Then $d = 4$, $X^+ = \Bl_p(V_{2,5})$ and $\X_t \simeq V_{2,4}$,
   \item $c_2(\F) = 4$. Then $d = 3$, $X^+$ is of the same type, and $\X_t \simeq V_{2,3}$,
   \item $c_2(\F) = 5$. Then $d = 2$, $X^+ \simeq X$, and $\X_t \to
     \PN_3$ is a double cover. 
  \end{enumerate}
\item If $\phi$ is birational, then $X = \Bl_p(Y)$ for a general point $p$ in a smooth del Pezzo threefold $Y = V_{2,d+1}$, such that
 \begin{enumerate}
   \item $d = 1$, $X^+ \simeq X$ and $\X_t
     \to W$ is a double cover of the Veronese cone,
   \item $d = 2$, $X^+ \simeq X$ and $\X_t \to \PN_3$
     is a double cover,
   \item $d = 3$, $X^+$ admits a del Pezzo fibration
     as in (1), (ii), and $\X_t \simeq V_{2,3}$,
   \item $d = 4$, $X^+ = \PN(\F)$ as in (2), (ii), and $\X_t \simeq V_{2,4}$. 
  \end{enumerate}
\end{enumerate}
\end{theorem2}

\noindent {\bf From now on we will assume for the rest of the paper
  that $r_X = 1$.}


\section{Del Pezzo fibrations}
\setcounter{lemma}{0}

In this section we consider almost Fano threefolds admitting a del Pezzo fibration.

\begin{setup} {\rm We fix for this section the following setup. $X$ is as always a smooth projective threefold with $-K_X$ big and nef, but not ample. 
Suppose that
$\phi: X \to \bP_1 $ is a del Pezzo fibration, which is the contraction of an extremal ray, i.e. $\rho (X) = 2.$ Let $F$ denote a general fiber of $\phi.$
Notice that $K_F^2 \ne 7$ by \cite{Mori}. We put $F' = \psi(F) $ and $F'' = \mu(F').$ Its strict transform in $X^+$ will be
called $\tilde F.$ \\
Since we assume that $X$ has index $1,$ (\ref{antmodel}) plus classification gives
$$ (-K_X)^3 \leq 22.$$
}
\end{setup}

The case that $\phi$ is a $\bP_2-$bundles was already treated in Part I (\cite{JPR}); here the only possible case is the small resolution of the quadric cone, which has index $r_X = 3$. This is already traeted in Propostion~\ref{highindex} above.

\

\noindent
{\bf From now on we shall assume that $F \ne \bP_2$ (for some or - equivalently - all fibers)}.

\begin{notation} \label{notations}
{\rm (1) We introduce the number $\lambda $ to be the maximal integer such that  
$$ H^0(-K_X - \lambda F) \ne 0.$$ 
 \\ 
(2) We recall the notations
\begin{equation} \label{not1}
 \tilde L^+ = \alpha (-K_X) + \beta F  
\end{equation}
and 
\begin{equation} \label{not2}
 \tilde F =  \alpha^+ (-K_{X^+}) + \beta^+ L^+. 
\end{equation}
If $\dim Y^+ = 1,$ then we shall write $F^+$ instead of $L^+.$ }
\end{notation} 

\begin{lemma} \label{div8}
If $K_F^2 = 8$, then $({-}K_X)^3$ is divisible by $8$.
\end{lemma}

\begin{proof}
By \cite{Mori}, $X \subset \PN(\sF)$ for some rank $4$ bundle $\pi: \sF \to
\PN_1$. Denote the tautological line bundle on $\PN(\sF)$ by
$\zeta$ and let $X \in |2\zeta + \pi^*\sO(\mu)|$ for some integer
$\mu$. Then
 \[-K_X = 2\zeta - \pi^*\sO(c_1+\mu-2),\]
where $c_1 = c_1(\sF)$, i.e. there exists some integer $b$ such that 
 \[L = \frac{1}{2}(-K_X + \pi^*\sO(b)) \in \Pic(X).\]
Now Riemann Roch for $L$ gives
 \[\chi(L) = 2 + 2b + \frac{({-}K_X)^3}{8}\]
proving the claim.
\end{proof}

\begin{proposition} Consider the number $\lambda $ introduced in (\ref{notations}). 
If $K_F^2 < 8,$  all members of $\vert -K_X- \lambda F \vert $ are irreducible and reduced. 
\end{proposition} 

\proof  Let $R \in \vert -K_X - \lambda F \vert$ and suppose that $R = R_1 + R_2$. Since $K_F^2 < 8,$ the del Pezzo surface 
$F$ contains $(-1)-$curves, hence say $R_2 \vert F \equiv 0$ (recall that $R_i \vert F$ are proportional since $\rho(X/Y) = 1$), 
so that $R_2 = \phi^*(\sO(a)),$ contradicting
the maximality of $\lambda.$
\qed

\begin{proposition} \label{alphabeta}
In the notations (\ref{notations}) the following holds. 
\begin{enumerate} 
\item $\beta \beta^+ = 1, \ \alpha + \beta \alpha^+ = \alpha^+ + \beta^+ \alpha = 0.$ 
\item If $ K_F^2 \leq 6$ and if there exists a rational curve $l^+ \subset X^+$ with $-K_{X^+} \cdot l^+ = 1$ and $L^+ \cdot l^+ = 0,$ then 
$\beta = \beta^+ = -1$ and $\alpha = \alpha^+ \in \bN.$ 
\item If $K_F^2 \leq 6$ and a curve $l^+$ as in (2) does not exist, then either (2) holds or $(\beta,\beta^+) = (-2,-{{1} \over {2}}).$ 
\item If $ K_F^2 = 8,$ then either (2) holds or $(\beta,\beta^+) = (-{{1} \over {2}},-2), (-2,-{{1} \over {2}}).$
\item Let $D \subset X$ be an irreducible effective divisor with strict transform $\tilde D \subset  X^+.$ Then 
$$ K_X^2 \cdot D = K_{X^+}^2 \cdot \tilde D $$
and $$ K_X \cdot D^2 = K_{X^+} \cdot \tilde D^2.$$
\end{enumerate}  
\end{proposition}

\proof 
(1) follows by inserting \eqref{not1} into \eqref{not2} and vice versa and by comparing coefficients. \\
(2) In the decomposition
$$ \tilde L^+ = \alpha (-K_X) + \beta F, \eqno (+)$$
the numbers numbers $\alpha, \beta$ are rational a priori. Suppose that $K_F^2 \leq 6.$ 
Let $l_{\phi}$ be a $(-1)-$curve in $F$ and intersect with (+):
$$ \tilde L^+ \cdot l_{\phi} = \alpha + 0, $$
hence $\alpha \in \bN.$ 
Thus $\beta F$ is Cartier so that $\beta \in \bZ.$ Similarly, the
existence of $l^+$ gives $\alpha \in \bN$ and $\beta^+ \in \bZ.$
Since $\beta < 0,$ the claim (2) follows. \\
(3) and (4) are done in the same way. We just observe that if there is no curve $l^+$ with $-K_{X^+} \cdot l^+ = 1$, then at least
we can find $l^+$ such that  $-K_{X^+} \cdot l^+ = 2.$  \\
(5) Finally for (5)  just represent the spanned line bundle $-K_X$ by general members not meeting the exceptional locus of $\psi.$  
\qed

\begin{proposition} If $\lambda = 0,$ then $\beta \ne  -{{1} \over {2}}.$ 
\end{proposition} 

\proof 
Assume $\lambda = 0$ and $\beta = -\frac{1}{2}$.
Observe that $\alpha \not \in \bN,$ because otherwise $F$ would be
divisible in ${\rm Pic}(X).$ Hence we find a line bundle $M$ such that
$$ -K_X = F + 2M.$$  
Since $\lambda = 0$ 
we have $H^0(-K_X-F) = 0$, and therefore $\phi_*(-K_X)$ has the form
$$ \phi_*(-K_X) = \sO^a \oplus \sO(-1)^b.$$ 
Since $a+b = K_F^2 + 1 = 9$ and 
$ a = h^0(-K_X) = {{(-K_X)^3} \over {2}} + 3,$ we obtain 
$ (-K_X)^3 = 2a - 6,$
and we must have $$(-K_X)^3 \leq 12.$$
The line bundle $M+F = {{1} \over {2}}(-K_X + F) $ is ample and by Kodaira vanishing and Riemann-Roch we obtain
$$ h^0(M+F) = \chi(M+F) = {{(-K_X)^3} \over {8}} + 4.$$ 
Hence $(-K_X)^3 = 8$ and $h^0(M+F) = 5.$ 
Consider the exact sequence
$$ 0 \to H^0(M) \to H^0(M+F) \to H^0((M+F) \vert F) = H^0(M \vert F). $$
Since $\lambda = 0,$ we have $H^0(M) = 0; $ moreover $M \vert F = {{-K_F} \over {2}}$, so that $h^0((M+F) \vert F) = 4.$
This contradicts the exact sequence. 
\qed 

\begin{proposition} \label{kx2kf}
Suppose that $\dim Y^+ = 1$ and that $\beta = -1.$ Then $\alpha (-K_X)^3 = 2 K_F^2.$ 
\end{proposition} 

\proof This is a consequence of  $K_X \cdot (\tilde F^+)^2 = 0$ and (\ref{alphabeta}). 
\qed 

\vskip .2cm \noindent
First we consider the ``exotic'' cases in \ref{alphabeta}.  

\begin{proposition} \label{beta}
Assume $\dim Y^+ = 1$ (and that $X^+$ has index
  $1$). Suppose $K_F^2 = 8.$ Then $\beta = -1$ except $K_{F^+}^2 = 4$,
  $({-}K_X)^3 = 16$, and $(\alpha^+, \beta^+) = (1, -2)$. This case
  really exists and $X^+ \subset \phi^+_*(-K_{X_+}) = \PN(\sO(2)^{\oplus 2} \oplus \sO(1)^{\oplus 2} \oplus \sO)$ may be realized as a complete intersection of two general sections in $|2\zeta -2F^+|$.
\end{proposition} 

\proof We may assume that $\beta = -{{1} \over {2}};$ the other case follows by interchanging the roles of $X$ and $X^+.$ 
Recall that $\alpha \not \in \bN,$ because otherwise $F$ would be
divisible in ${\rm Pic}(X).$ Hence we find a line bundle $M$ such that
$$ -K_X = F + 2M.$$  
By (3.?), $\lambda \geq 1.$ 
By cubing the equation $-K_X - F = 2M,$ we obtain $(-K_X)^3 = 8,16$ with $M^3 = -2, -1.$ 
Then Riemann-Roch gives $\chi(M) \geq 1.$ Since $h^q(M) = 0$ for $q \geq 2$ (apply the Leray spectral sequence to $\phi: X \to Y = \bP_1$), we have
$$ h^0(M) \geq 1.$$ 
Recall 
$$ \tilde F^+ = \alpha (-K_X) -{{1} \over {2}}F$$
with $\alpha$ a half-integer. 
Thus $2 \tilde  F^+ = 2\alpha (-K_X) - F$, and from $h^0(2 \tilde F^+)
= 3$ and $\lambda \geq 1$ we deduce $\alpha = {{1} \over {2}}$. We obtain 
$M = \tilde F^+$ and hence $\lambda^+ = 2$.
The  equation $K_X \cdot (\tilde F^+)^2 = 0$ gives 
$$ \alpha (-K_X)^3 = K_F^2 = 8, $$
so that $(-K_X)^3 = 16.$ 
Dually, $K_{X^+} \cdot (\tilde F)^2 = 0$ yields
$$ \alpha^+ (-K_X)^3 = 4 K_{F^+}^2,$$
and with $\alpha^+ = 2 \alpha,$ we get $K_{F^+}^2 = 4.$  
This case in fact exists: First note $\lambda^+ = 2$ and $h^0(X^+,
-K_{X^+} -2F^+) = h^0(X, F) = 2$. Let
 \[\sE = \phi^+_*(-K_{X^+}) = \sO(2)^2 \oplus \sO(1)^a + \sO^b +
 \sO^c.\]
Then $h^0(-K_{X^+}) = 11$ and $K_{F^+}^2 = 4$ gives $2a+b = 5$ and
$a+b+c = 3$, hence $a = 2, b = 1$ and $c = 0$. Take $X^+ \subset
\PN(\sE)$ a complete intersection of two sections 
 \[Q_i \in |\zeta -2F^+|, \quad i=1,2\]
where the fiber of $\PN(\sE)$ is denoted by $F^+$ as well. For $Q_i$
general then $X^+$ is smooth with $-K_{X^+} =
\sO_{\PN(\sE)}(1)|_{X^+}$ is big and nef. Moreover, the exceptional
curve $C_0$  corresponding to the only trivial summand of $\sE$ is
contained in $Q_1$ and $Q_2$, hence $C_0 \subset X^+$ is the
exceptional locus of $\psi^+$. By construction, $X^+$ admits a del
Pezzo fibration with $K_{F^+}^2 = 4$.

Concerning the flop note first that $N_{C_0/X^+}$ is of type
$(-1,-1)$, hence $X^+ \dasharrow X$ is a simple flop. The linear
system 
 \[|\zeta -2F^+|\]
on $\PN(\sE)$ defines a rational map onto $\PN_1$ with base locus a
threefold $Z$ containing $C_0$. It is easy to see that $Z \cap Q_1
\cap Q_2 = C_0$, hence $X$ admits a del Pezzo fibration with $K_F^2 =
(-K_X -2F^+)\cdot K_{X^+}^2 = 8$. 
\qed 


\begin{proposition} \label{dpdp}
Suppose $\dim Y^+ = 1$. Then either
 \begin{enumerate}
  \item $\lambda = 1$ and $(({-}K_X)^3, K_F^2) = (2,1)$, $(4,2)$, $(6,3)$, $(10,5)$, $(12,6)$, where the first three cases definitely exist; or
 \item $\lambda = 0$ and $(({-}K_X^3), K_F^2) = (2,3)$, $(2,4)$, $(2,5)$, $(2,6)$, $(4,4)$, $(4,6)$, $(6,6)$, where the cases $(2,3)$, $(2,4)$ and $(4,4)$ definitely exist.
 \end{enumerate}
The existence of the remaining cases is open.
\end{proposition} 

\proof By Prop. \ref{beta} we may assume $\beta = \beta^+ = -1$ and
$\alpha = \alpha^+ \in \bN$ in (\ref{alphabeta}). So 
$$ \tilde F^+ = \alpha(-K_X) - F.$$
Since all members of $\vert \tilde F^+ \vert $ are irreducible, we must have $\lambda \leq 1$ and if $\lambda = 1,$ then $\alpha = 1.$ 
\vskip .2cm \noindent
{\bf Case I:} $\lambda = 1.$ \\
Then $ -K_X = F + \tilde F^+$ and we obtain
$$ \sE = \phi_*(-K_X) = \sO(1)^2 \oplus \sO^b \oplus \sO(-1)^c $$
with $2+b+c = K_F^2+1$ and $4+b = h^0(-K_X).$ Since $(-K_X)^3 = 2K_F^2 $ by (\ref{kx2kf}), it follows $c = 0$ so that
$$ H^1(-K_X-F) = H^1(\tilde F^+) = 0.$$ 
Thus 
$$((-K_X)^3,K_F^2) = (2,1),(4,2),(6,3),(8,4),(10,5),(12,6),(16,8).$$ 
We will continue case by case.

\vspace{0.3cm}

\noindent {\bf 1.) $({-}K_X)^3 = 2$, $K_F^2 = 1$.} Here $|{-}K_X|$ has base points; this case exists and was already treated in Proposition~\ref{basepoints}.

\vspace{0.3cm}

\noindent {\bf 2.) $({-}K_X)^3 = 4$, $K_F^2 = 2$.} This case exists, here $X \to \PN(\sE)$ with $\sE = \sO \oplus \sO(1)^2$ is a double covering, i.e. $X$ is hyperelliptic. The construction can be found in Lemma~\ref{hypsymm}.

\vspace{0.3cm}

\noindent {\bf 3.) $({-}K_X)^3 = 6$, $K_F^2 = 3$.} We construct this case as follows. Let $\sE = \sO^2 \oplus \sO(1)^2$ and define
 \[X \in |3\zeta|\]
general. Then $-K_X = \zeta|_X$, hence $X$ is a smooth almost Fano threefold with $({-}K_X)^3 = 6$. The general fiber of the induced map $X \to \PN_1$ is a cubic in $\PN_3$, hence a del Pezzo surface of degree $3$. The map 
 \[\nu: \PN(\sE) \lra \PN_5\]
given by $|\zeta|$ contracts the surface $S \simeq \PN_1 \times \PN_1$ corresponding to the two trivial summands of $\sE$ along a ruling to a line $C \subset \PN_5$. Since $X \in |\nu^*\sO_{\PN_5}(3)|$ by construction, $X \cap S$ consists of $3$ fibers of $\nu$, i.e. the anticanonical map $\psi: X \to X'$ contracts three smooth rational curves to points. 

Note that the linear system $|{-}K_X-F|$ is a pencil with base locus the exceptional locus of $\psi$. This shows $X^+$ again admits a del Pezzo fibration. 

\vspace{0.3cm}

\noindent {\bf 4.) $({-}K_X)^3 = 8$, $K_F^2 = 4$.} The construction is analogously to the last case: define $X \subset \PN(\sE)$ with $\sE = \sO^3 \oplus \sO(1)^2$ as a complete intersection of two general elements in $|2\zeta|$. Then $X$ admits a del Pezzo fibration with $K_F^2 = 4$ and the anticanonical map contracts $4$ smooth rational curves. The flop is of the same type as above.

\vspace{0.3cm}

\noindent {\bf 5.) $({-}K_X)^3 = 10$, $K_F^2 = 5$.} Open.

\vspace{0.3cm}

\noindent {\bf 6.) $({-}K_X)^3 = 12$, $K_F^2 = 6$.} Open.

\vspace{0.3cm}

\noindent {\bf 7.) $({-}K_X)^3 = 16$, $K_F^2 = 8$.} This case does not exist for the following reason. By Proposition~\ref{antmodel}, $X'$ admits a smoothing $\X$ such that $\X_t$ is a smooth Fano threefold of index $1$ with $({-}K_X)^3 = 16$. Then $\X_t$ contains lines $l_t$ and the degeneration of $l_t$ to $X$ gives a line $l_0$ in $X$. Let $\hat{l_0} \subset X$ be the strict transform of $l_0$. Then 
 \[-K_X \cdot \hat{l_0} = 1.\] 
On the other hand, $-K_X = F + \tilde{F}^+$ by assumption. Since $-K_X|_F = K_F$ is divisible by two, $\hat{l}_0$ cannot be contained in a fiber. This shows $F\cdot\hat{l}_0 > 0$. Then $\tilde{F}^+\cdot \hat{l}_0 \le 0$. But $\hat{l}_0$ is not $\psi$-exceptional, hence $\tilde{F}^+\cdot \hat{l}_0 = 0$. Then the strict transform of $\hat{l}_0$ in $X^+$ is contained in the fiber $F^+$, which is impossible as above.

\vskip .2cm \noindent 
{\bf Case II:} $\lambda = 0.$ \\ 
So $\alpha (-K_X) = F + \tilde F^+ $ with $\alpha \geq 2.$  
Here $\sE = \phi_*(-K_X)$ has the form  
$$ \sE = \sO^a \oplus \sO(-1)^b.$$  
Since $a+b = K_F^2 + 1$ and  
$ a = h^0(-K_X) = {{(-K_X)^3} \over {2}} + 3,$ we obtain   
$$ (-K_X)^3 = 2a - 6.$$
Thus $a \geq 4.$ Now $\alpha (-K_X)^3 = 2K_F^2$ hence only the following cases remain.  
\begin{enumerate}
\item $a = 4, (-K_X)^3 = 2$ and $(K_F^2,\alpha)  = (3,3),(4,4),(5,5),(6,6),(8,8);$ 
\item $a = 5, (-K_X)^3 = 4 $ and $(K_F^2,\alpha) = (4,2), (6,3),(8,4);$
\item $a = 6, (-K_X)^3 = 6$ and $(K_F^2,\alpha) = (6,2);$ 
\item $a = 7, (-K_X)^3 = 8$ and $(K_F^2,\alpha) = (8,2).$ 
\end{enumerate} 
We will continue case by case.

\vspace{0.3cm}

\noindent {\bf 1.) $a=4$, $K_F^2=3$.} We have $a=4$ and $b=0$, hence take 
 \[X \subset \PN(\sO^{\oplus 4}) \simeq \PN_3\times \PN_1\]
a general element in $|3\zeta + 2F| = |(3,2)|$. Then $X$ is a smooth almost Fano threefold with the expected numerical data by adjunction formula. The anticanonical map is the restriction of the second projection 
 \[\pi: \PN(\sO^{\oplus 4}) \lra \PN_3\] 
defined by $|\zeta|$. Since $X\cdot l_{\pi} = 2$ for a general fiber of $\pi$, the restriction of $\pi$ to $X$ factors, i.e., $X$ is hyperelliptic with
 \[X \stackrel{\psi}{\lra} X' \stackrel{\mu}{\lra} \PN_3,\]
$\mu$ a double covering. It remains to show that $\psi$ is small. Let $x_0, \dots, x_3, y_0, y_1$ be homogeneous coordinates of the product $\PN_3\times \PN_1$. Then $X$ is defined by some $f(x_0, \dots, x_3, y_0, y_1)$, homogeneous of degree $3$ is the $x_i$ and of degree $2$ in the $y_i$. Write
 \[f=g_0y_0^2+g_1y_0y_1+g_2y_1^2, \quad g_i \in \KC[x_0, \dots, x_3]\]
homogeneous of degree $3$. Since $X$ was taken general, the $g_i$ are general. Then the fiber over some point $a=[a_0:\dots:a_3] \in \PN_3$ is contained in $X$, iff $g_0(a)=g_1(a)=g_2(a)=0$. This shows $X$ contains fibers of $\pi$ exactly over the complete intersection $g_0=g_1=g_2=0$ in $\PN_3$, which are $27$ points. Hence $\psi$ is small with $27$ exceptional curves. Since $X$ is hyperelliptic, the flop $X^+$ is of the same type as $X$. 

\vspace{0.3cm}

\noindent {\bf 2.) $a=4$, $K_F^2=4$.} We have $a=4$ and $b=1$, hence take 
 \[X \subset \PN(\sO^{\oplus 4} \oplus \sO(-1))\]
a complete intersection of two general elements
 \[Q_1 \in |2\zeta+F|, \quad Q_2 \in |2\zeta+2F|.\]
The base locus $C_0 = \Bs|\zeta|$ has $C_0\cdot Q_2 = 0$. But the system $|2\zeta+2F$ is base point free, hence the general element $Q_2$ does not contain $C_0$. This shows $X$ is a smooth almost Fano threefold with the expected numerical data. It remains to show $\psi: X \to X'$ is small. This can be done either directly in coordinates as in the last case, or just by checking the respective lists for the divisorial case in \cite{JPR}.

\vspace{0.3cm}

\noindent {\bf 3.) $a=4$, $K_F^2=5$.} Open.

\vspace{0.3cm}

\noindent {\bf 4.) $a=4$, $K_F^2=6$.} Open.

\vspace{0.3cm}

\noindent {\bf 5.) $a=4$, $K_F^2=8$.} Then $({-}K_X)^3 = 2$, which is impossible due to Lemma~\ref{div8}.

\vspace{0.3cm}

\noindent {\bf 6.) $a=5$, $K_F^2=4$.} Here $b=0$, hence 
 \[X \subset \PN(\sO^{\oplus 5}) \simeq \PN_4 \times \PN_1\]
has codimension $2$. Since $({-}K_X)^3 = 4$, either $X' \subset \PN_4$ is a quartic, or a double covering of a quadric, i.e. hyperelliptic. Both cases do exist.

(a) Take for $X$ the complete intersection of two general elements
 \[Q_0, Q_1 \in |2\zeta+F| = |(2,1)|.\]
Then $X$ is smooth almost Fano. Let $x_0, \dots, x_4, y_0, y_1$ be homogeneous coordinates of $\PN_4\times \PN_1$. Then $Q_i$ is defined by
 \[y_0q_{i0} + y_1q_{i1} = 0, \quad q_{ij} \in \KC[x_0, \dots, x_4]\]
general quadrics. The image of $\psi: X \to \PN_4$ is given by the determinant of
 \[Q = \left(\begin{array}{ll} q_{00} & q_{01}\\q_{10} & q_{11}\end{array}\right),\]
hence a quartic in $\PN_4$. We have exactly one point in $X$ over any $p \in \PN_4$ with $\rk Q(p) = 1$, and a whole $\PN_1$ over all $p$ with $\rk Q(p) = 0$. But $\rk Q(p) = 0$ means $q_{ij}(p) = 0$ for all $i,j$, hence $X \to X'$ has exceptional fibers over the intersection of the $4$ general quadrics $q_{ij}$ in $\PN_4$ cutting out $16$ points. This shows $\psi$ is small with $16$ exceptional fibers.

Concerning the flop we consider the linear system $|{-}2K_X-F| = |(2,-1)|_X|$ with base locus exactly $\exc(\psi)$. Chasing successively the twisted ideal sequences of $X \subset Q_0 \subset \PN_4 \times \PN_1$ we find $h^0(X, (2,-1)|_X) = 2$, i.e. the flop $X^+$ again admits a del Pezzo fibration and is in fact of the same type as $X$.

\vspace{0.3cm}

(b) Take for $X$ the complete intersection of
 \[Q_0 \in |2\zeta| = |(2,0)|, \quad Q_1 \in |2\zeta+2F| = |(2,2)|.\]
As above, the $Q_i$ are given by
 \[q_0 = 0, \quad y_0^2q_1 + y_0y_1q_2+y_1^2q_3 = 0,\]
respectively. Now the image of $X$ in $\PN_4$ is the quadric $Q_0$, the general fiber consists of two points, and we have again $16$ exceptional fibers. Since $X$ is hyperelliptic, the flop $X^+$ is of the same type as $X$. 

\vspace{0.3cm}

\noindent {\bf 7.) $a=5$, $K_F^2=6$.} Open.

\vspace{0.3cm}

\noindent {\bf 8.) $a=5$, $K_F^2=8$.} Here $({-}K_X)^3 = 4$, contradicting Lemma~\ref{div8}.

\vspace{0.3cm}

\noindent {\bf 9.) $a=6$, $K_F^2=6$.} Open.

\vspace{0.3cm}

\noindent {\bf 10.) $a=7$, $K_F^2=8$.} This case does not exist by the following argument. By \cite{Mori}, $X \subset \PN(\sF)$ for some rank $4$ vector bundle $\sF$ on $\PN_1$ and
 \[X \in |2\zeta + \pi^*\sO(\mu)|\]
for some integer $\mu$. By adjunction formula, we have 
 \[-K_X = 2\zeta + (2-c_1-\mu)F\]
with $c_1 = c_1(\sF)$. Assume $\sF$ is normalized such that $-3 \le c_1 \le 0$. Then 
 \[8 = ({-}K_X)^3 = -8c_1-16\mu+48,\]
hence $-K_X = 2\zeta + (-\frac{c_1}{2} - \frac{1}{2})F$. Since $-K_X$ is a line bundle and not divisible in $\Pic(X)$ by assumption, we must have $c_1 = -3$ and $\mu = 4$.

Now $\alpha = 2$ gives $\tilde{F}^+ = -2K_X - F$, i.e.
 \[\tilde{F}^+ = 4\zeta|_X + F.\]
Hence $h^0(X, \zeta|_X) = 0$. The twisted ideal sequence of $X$ shows $\sF = \pi_*\zeta = \pi_*(\zeta|_X)$, i.e.
 \[H^0(\PN_1, \sF) = 0.\]
Assume $\sE = \oplus_{i=1}^4 \sO(a_i)$. Then $\sum a_i = -3$ and $a_i < 0$ for all $i$. This is impossible.
\qed

\begin{proposition} \label{delpezzoconicnum}
Assume $\dim Y^+ = 2$ and let $\tau^+$ be the degree of the discriminant locus of $\phi^+.$ 
Then, using the notations of (\ref{notations}),  
\begin{enumerate} 
\item Either $(\beta, \beta^+) = (-1,-1) $ or  $(\beta, \beta^+) = (-{{1} \over {2}},-2).$ 
\item $K_X \cdot (\tilde L^+)^2 = -2.$
\item $\alpha^2(-K_X)^3 = 2\alpha(12-\tau^+) - 2$.
\item $\alpha^2(-K_X)^3 = 2\alpha K_F^2  + 2$ if $\beta = -1;$ otherwise \\
$\alpha^2(-K_X)^3 = \alpha K_F^2 + 2.$
\end{enumerate}
\end{proposition}

\proof (1) follows from (\ref{alphabeta}): the case $\phi^+$ is a
$\PN_1$-bundle is treated in section~\ref{secconic}, we may hence
assume $\phi^+$ is a proper conic bundle. i.e. there are singular fibers and therefore there exists a curve
$l^+$ with $-K_{X^+} \cdot l^+ = 1.$ 
 \\
(2) This is (\ref{alphabeta}), (5). 
\\
(3) is a consequence of (1) and (3.7), having in mind that $(-K_X)^3 = (-K_{X^+})^3.$ \\
(4) is a consequence of (1)  and (3).
\qed

\begin{proposition} \label{delpezzoconic}
If $\dim Y^+ = 2,$ then $\beta = -1$ and $(({-}K_X)^3, K_F^2, \tau^+)$ is one of $(8,3,7)$, $(10,4,6)$, $(12,5,5)$, $(14,6,4)$, where the first two cases really exist.
\end{proposition}

\proof (1) We first consider the case that $\beta = -1.$ Then
\ref{delpezzoconicnum}(3) and (4) give
$$ \alpha(12 - \tau^+ - K_F^2) = 2.$$
Hence either
$$ \alpha = 1; \  12 - \tau^+ - K_F^2 = 2,$$
or
$$ \alpha = 2; \ 12 - \tau^+ - K_F^2 = 1.$$
The second alternative however contradicts \ref{delpezzoconicnum}(4).
So $\alpha = 1$. 
Here $(-K_X)^3 = 2K_F^2 + 2$ and $-K_X = F + \tilde L^+$. So $h^0(-K_X - F) = 3$ and $\vert -K_X - F\vert $ does not
contain reducible members. This implies
$$ \sE = \phi_*(-K_X) = \sO(1)^3 \oplus \sO^b \oplus \sO(-1)^c $$
with $3+b+c = K_F^2 + 1$ and $h^0(-K_X) = 6+b.$ Hence $c = 0$ and 
$$ K_F^2 = b+2; \ (-K_X)^3 = 2b+6.$$ 
Next we consider the spanned rank 3 bundle
$$ \sE^+ = \phi^+_*(-K_{X^+}).$$ 
Using the notations of section~\ref{secconic}, $X^+ \subset \bP(\sE^+)$ is a divisor of the form
$$ [X] = 2\zeta + \pi^*(\sO(\lambda)), \quad \lambda = 3-c_1,$$
such that $\zeta \vert X^+ = -K_{X^+}.$ Let $c_i = c_i(\sE^+).$
Then
$$ 2b+6 = (-K_X)^3 = \zeta^3 \cdot X = c_1^2 - 2c_2 + 3c_1. \eqno (*)$$
The equation $K_{X^+}^2 \cdot L^+ = 12 - \tau^+ $ translates into 
$$ c_1 = 9 - \tau^+. \eqno(**) $$  
>From $(-K_X)^3 = 2b+6 $ and \ref{delpezzoconicnum}(3) we get 
$$b = 8-\tau^+,$$
in particular $2 \leq \tau^+ \leq 8$, $\tau^+\not= 3$. 
Putting this and (**) into (*) gives
$$ (\tau^+)^2 - 19 \tau^+ + 86 = 2c_2.$$ 
We consider a general section in $\sE^+$ and obtain a rank 2 vector bundle $\sF^+$ from the exact sequence
$$ 0  \to \sO \to \sE^+ \to \sF^+ \to 0.$$ 
Hence $h^0(\sF^+(-1)) = h^0(\sE^+(-1)) = 2.$ We continue case by case.

\vspace{0.3cm}

\noindent {\bf 1.) $\tau^+ = 8.$} Then $c_1(\sF^+) = 1$, $c_2(\sF^+) = -1$, which is impossible.

\vspace{0.3cm}

\noindent {\bf 2.) $\tau^+ = 7.$} Then $K_F^2 = 3$ and 
 \[X \subset \PN(\sE) = \PN(\sO(1)^{\oplus 3} \oplus \sO)\]
is a hypersurface. Take $X \in |3\zeta -F|$ general. Then $X$ is a smooth almost Fano threefold with $-K_X = \zeta|_X$. The map
 \[\PN(\sE) \stackrel{|\zeta|}{\lra} Z \subset \PN_6\]
is a small resolution of the cone over $\PN_1 \times \PN_2 \hookrightarrow \PN_5$ (Segre enbedding). The exceptional curve $C_0 \subset \PN(\sE)$ corresponds to the projection $\sE \to \sO$, the only trivial summand of $\sE$. Since
 \[X\cdot C_0 = (3\zeta-F)\cdot C_0 = -1,\]
$X$ contains $C_0$, hence $\psi$ is small with exactly one exceptional curve. 

Concerning the flop note that the normal bundle of $C_0$ in $X$ is of type $(-1,-1)$, i.e. $X \dasharrow X^+$ is a simple flop. The linear system
 \[|(\zeta-F)|_X|\]
has base locus exactly $C_0$ and we find $h^0(X, \zeta-F) = 3$. This shows $X^+$ admits a conic bundle structure and the remaining data are easily verified.

\vspace{0.3cm}

\noindent {\bf 3.) $\tau^+ = 6.$} Then $K_F^2 = 4$ and
 \[X \subset \PN(\sE) = \PN(\sO(1)^{\oplus 3} \oplus \sO^{\oplus 2}),\]
a small resolution of the double cone $Z \subset \PN_7$ over $\PN_1 \times \PN_2 \hookrightarrow \PN_5$ embedded by Segre, i.e.
\[\PN(\sE) \stackrel{|\zeta|}{\lra} Z \subset \PN_7.\]
Take $X$ a complete intersection of
 \[Q_1 \in |2\zeta|, \quad Q_2 \in |2\zeta-F|\]
general. Then $X$ is a smooth almost Fano threefold with $-K_X = \zeta|_X$. The first quadric $Q_1$ is the pullback of some general quadric in $\PN_7$ intersecting the vertex of the cone $Z$ in two points. This means
 \[Q_1 \lra Z_1 \subset \PN_7\]
is birational with exceptional locus two smooth rational curves $C_1, C_2 \subset Q_1$. The threefold $X$ is a divisor on $Q_1$, cut out by $Q_2$, and we find
 \[X\cdot(C_1+C_2) = Q_1\cdot Q_2 \cdot S = -2,\]
where $S \simeq \PN_1 \times \PN_1 = \PN(\sO^{\oplus 2})$ is the exceptional surface of $\PN(\sE) \to Z$. Since the two curves are numerically equivalent, both have negative intersetion number with $X$ in $Q_1$, are hence contained in $X$. This shows $\psi$ is small with two smooth exceptional curves.

Concerning the flop consider as above the linear system $|\zeta-F|$ on $X$, which has base locus $C_1 \cup C_2$ on $X$ and admits $3$ sections. The remaining data of $X^+$ now follow numerically.

\vspace{0.3cm}

\noindent {\bf 4.) $\tau^+ = 5.$} Open.

\vspace{0.3cm}

\noindent {\bf 5.) $\tau^+ = 4.$} Open. 

\vspace{0.3cm}

\noindent {\bf 6.) $\tau^+ = 3.$} Then $K_F^2 = 7$, which is impossible.

\vspace{0.3cm}

\noindent {\bf 7.) $\tau^+ = 2.$} Then $K_F^2=8$ and $({-}K_X)^3 = 18$, contradicting Lemma~\ref{div8}.

\vspace{0.3cm}

\noindent {\bf 8.) $\tau^+ = 1.$} Then $K_F^2 = 9$, which is ruled out by assumption.

\vskip .2cm \noindent   
(2) Now consider the case $\beta = -{{1} \over {2}}$ so that $\beta^+ = -2$ and $K_F^2 = 8.$ 
Arguing in the same way as in (1), we get
$$ \alpha(48 - 4 \tau^+ - K_F^2) = 6.$$ 
Together with \ref{delpezzoconicnum}(4) this yields a contradiction. 
\qed

\begin{setup} {\rm  
Assume that $\phi^+$ is birational. The exceptional divisor will be denoted $E^+$ and its strict transform in $X$ by 
$\tilde E^+.$ Slightly differing from (\ref{notations}), we will substitute $L^+$ by $E^+$ and shall write
\begin{equation} \label{setup1}
 \tilde F = \alpha^+ (-K_{X^+}) + \beta^+ E^+ 
\end{equation}
and 
\begin{equation} \label{setup2} 
\tilde E^+ = \alpha (-K_X) + \beta F. 
\end{equation}
All results of Proposition \ref{alphabeta} remain valid. Denote the generator of $\Pic(Y^+)$ by $H^+$, i.e., $L^+ = (\phi^+)^*H^+$. If $Y^+$ is smooth, then let $-K_{Y^+} = rH^+$, $1 \le r \le 4$ the index of $Y^+$.}
\end{setup}

\begin{proposition} \label{blowupsmooth}
Suppose  $E^+ = \bP_2$ with normal bundle $\sO(-1)$. Then $Y^+$ has index $1$ and $(-K_{Y^+})^3 = 18.$ 
Moreover $(-K_X)^3 = 10$ and $K_F^2 = 6.$ 
\end{proposition} 

\proof Suppose that $\phi^+$ is the blow-up of the smooth point $p$ in the Fano threefold $Y^+.$  
First notice that $Y^+$ is a smooth Fano threefold with index $1.$ In fact, if $Y^+$ had index 2, then $X$ had index 2, which we ruled out.
If $Y^+$ had index 3 or 4, then $X^+$ would be Fano. By intersecting \eqref{setup1} with a general line in $Y^+$ not meeting $\phi^+(E^+),$ we get
$\alpha^+ \in \bN,$ hence $-\beta^+ \in \bN.$ We also notice that
$$ K_X \cdot (\tilde E^+)^2 = 2. \eqno (*)$$ 
\vskip .2cm \noindent (1) Suppose that $- \beta \in \bN$. Then $\beta = \beta^+ = -1.$ 
Combining (*) with \eqref{setup2} gives
$$ \alpha(\alpha K_X^3 + 2K_F^2) = 2,$$
hence $\alpha = 1,2$. If $\alpha = 2,$ then $2(-K_X)^3 = 2K_F^2  + 1$, so that only $\alpha = 1$ remains. 
Next we combine $K_{X^+} \cdot \tilde F^2 = 0$ with \eqref{setup1}, so that 
$$ \alpha^2 (-K_X)^3 + 8 \alpha = -2, $$
hence $$ (-K_X)^3 = 10, \ K_F^2 = 6.$$ 
Conversely, this case really exists by \cite{Ta89}. 
\vskip .2cm \noindent
(2) If $-\beta \not \in \bN,$ then $\beta = -{{1} \over {2}}$ and $\beta^+ = -2.$ Here
we obtain from $\tilde E^+ = {{1} \over {2}} \alpha^+ (-K_X) - {{1} \over {2}}F$ and (*) that
$$ \alpha^+(\alpha^+K_X^3 + 2K_F^2) = 8.$$
Since $\alpha^+$ is an odd positive integer, we conclude $\alpha^+ = 1.$ 
Then equation \eqref{setup1} gives $(-K_X)^3 = 24,$ hence $({-}K_{Y^+})^3 = 32$, which is impossible by classification.
\qed

\begin{proposition} \label{blowupq} 
Suppose $E^+$ is a quadric, either smooth or a quadric cone. Then
$\beta = \beta^+ =-1$, $\alpha = \alpha^+ = 1$, $({-}K_X)^3 = 6$, $K_F^2 = 4$, $({-}K_{Y^+})^3 = 8$ and $X
\subset \PN(\sO(1) \oplus \sO^4)$ may be realized as a
complete intersection.
\end{proposition} 

\proof By (\ref{hyp}) and (\ref{hypsymm}) we may assume that $X'$ is not hyperelliptic. Since $-K_{X^+} \vert E^+ = \sO(1),$ necessarily 
$$ H^0(-K_{X^+}) \to H^0(-K_{X^+} \vert E^+) $$
is surjective, since $\psi^+ \vert E^+$ must have degree 1. Thus $\psi^+ \vert E^+$ is actually an isomorphism. Consequently, if $l^+$ is a curve
contracted by $\psi^+,$ then 
$$E^+ \cdot l^+ = 1.$$ 
Since $-K_{X^+} \cdot l^+ = 0,$ \eqref{setup1} implies $-\beta^+ \in \bN.$ Intersecting \eqref{setup1} with a line in $E^+$ we furthermore see that $\alpha^+ - \beta^+ 
\in \bN,$ hence $\alpha^+ \in \bN.$ \\
(1) Suppose first that $-\beta \in \bN.$  
Then \eqref{setup2} gives $\beta = -1 = \beta^+ $ and $\tilde F \cdot l^+ = -1.$ It is also clear that $\alpha^+ = 1$ since $h^0(-K_{X^+} - E^+) \ne 0.$  
So $ h^0(-K_X) = 6$, so that $(-K_X)^3 = 6$ and $K_F^2 = 4;$ furthermore $(-K_{Y^+})^3 = 8.$ 

This case really exists: We have $h^0(-K_X-F) = 1$, hence 
 \[X \subset \PN(\phi_*(-K_X)) = \PN(\sE) = \PN(\sO(1) \oplus
 \sO^{\oplus 4}) = \Bl_{\PN_3}(\PN_5).\]
To realize $X$ take a complete intersection of
$Z_1 \in |2\zeta|$ and $Z_2 \in |2\zeta +F|$ general. Then $X$ is
smooth with $-K_X = \zeta|_X$. The anticanonical map $\psi: X \to X'$
is induced by the map 
 \[p: \PN(\sE) \lra \PN_5\] 
contracting $D \simeq \PN_1 \times \PN_3$ to $\PN_3$. The intersection
$E = D \cap X$ is a smooth surface, mapped by $p$ birationally to a
quadric in $\PN_3$ with exceptional locus $8$ rational curves. After
flopping these curves, the strict transform $E^+$ of $E$ becomes a
contractible quadric.
\vskip .2cm \noindent
(2) If $-\beta \not \in \bN,$ then $\beta = -  {{1} \over {2}}$ and $\beta^+ = -2.$ 
Arguing as in \ref{blowupsmooth}, (2), we obtain 
$$ \alpha^+ = 1, K_X^3 + 2K_F^2 = 8.$$
On the other hand, $K_{X^+} \cdot \tilde F^2 = 0$ leads to $(-K_X)^3 = 16,$ so that $K_F^2 = 12,$ which is absurd.   
\qed  

\begin{proposition} \label{blowupsing}
Suppose $E^+ = \bP_2$ with normal bundle $\sO(-2).$ Then $X \subset \bP(\sO_{\bP_1}(1) \oplus \sO_{\bP_1}^{\oplus 3}) = \bP(\sE)$
is a smooth element of $\vert 3\zeta + \pi^*\sO(1)\vert;$ and the flop of $X$ really has a contraction contracting a 
$\bP_2$ with normal bundle $\sO(-2).$
\end{proposition} 

\proof The first part of the proof is parallel to the proof of 3.14. If $-\beta \in \bN,$ we end up with $(-K_X)^3 = 4$ and $K_F^2 = 3$ and necessarily
$X \subset \bP(\phi_*(-K_X)) $ with $\phi_*(-K_X) = \sO(1) \oplus \sO^3$ over $\bP_1$ is in the linear system as stated in the proposition. \\
If conversely - in the obvious notation - $X \in \vert 3 \zeta + F \vert $ is a smooth element, then $-K_X = \zeta \vert X$ and it is easily 
checked that the blow-down 
$$ \bP(\sE) \to \bP_4$$ defined
by $\zeta$ restrict to a small contraction on $X$. 
It remains to show that $\phi^+$ really contracts a plane with normal
bundle $\sO(-2).$ This can be checked directly: let $D \subset
\PN(\sE)$ be the exceptional divisor and $E = X \cap D$. Then $E$ is a
smooth surface and the projection map $D \to \PN_2$ restricts to a
birational map on $E$, contracting $9$ rational curves. These curves
are the exceptional locus of $\psi: X \to X'$ and are of type
$(-1,-1)$. After flop, the strict transform $E^+$ of $E$ becomes a
contractible $\PN_2$. Following the restriction $-K_X|_E$ through the
flop diagram we find $-K_{X^+}|_{E^+} = \sO(1)$, hence $N_{E^+/X^+} =
\sO(-2)$ as claimed.

\vskip .2cm \noindent
If $-\beta \not \in \bN,$ then $-\beta = {{1} \over {2}}$ and by the same computations as in (\ref{blowupq}) we obtain $\alpha^+ = 1$ and then
that $({-}K_X)^3 = 12$ and hence $K_F^2 = 10$ which is absurd.
\qed 


\begin{proposition} \label{dpcurve}
 Suppose $\dim \phi^+(E^+) = 1.$ Then $X^+ = \Bl_{C^+}(Y^+)$ the blowup of a smooth Fano threefold $Y^+$ along a smooth curve $C^+$ by \cite{Mori} and $X$ is one of the cases listed in table~\ref{tdpcurve} in the appendix.
\end{proposition} 

\proof (1) Let $l^+$ be a curve contracted by $\phi^+.$ Intersecting \eqref{setup1} with $l^+$ gives
$$ \alpha^+ - \beta^+ \in \bN.$$
Choose a line $C' \subset Y^+$. If $r = 1,$ we can choose $C' $ disjoint from $C^+$. In fact, suppose that all lines meet $C^+$ and consider
their strict transforms $C$ in $X^+.$ Then $K_{X^+} \cdot C = 0$ so that $\psi^+$ would not be small. Hence the general line $C'$ is disjoint from
$C^+.$ Then we intersect \eqref{setup1} with $C'$ and obtain $\alpha^+ \in \bN$, hence $-\beta^+ \in \bN.$ \\
In the other cases we simply get 
$$ r \alpha^+ \in \bN,$$
hence $r(-\beta)^+ \in \bN,$ too. 

\vskip .2cm \noindent (2) Now the reasoning of (\ref{alphabeta})
applies: we have $\beta\beta^+ = 1$ and $\alpha + \alpha^+\beta =
0$. If $K_F^2 \leq 6,$ then $\alpha, \beta \in \KZ$. The case $\alpha, \beta \not\in \KZ$ may only happen if $K_F^2 =
8$, then $\alpha = \frac{\tilde{\alpha}}{2}, \beta =
\frac{\tilde{\beta}}{2}$ with $\tilde{\alpha}, \tilde{\beta} \in \KZ$ odd. Define
 \[\alpha', \beta' := \left\{\begin{array}{ll} \alpha, \beta & \mbox{ if
     } \alpha, \beta \in
     \KZ\\ 2\alpha, 2\beta  & \mbox{ if } \alpha, \beta \not\in
     \KZ\end{array}\right.\]
Then $r\beta^+
\in \KZ$ implies
 \begin{equation} \label{betar}
  \beta' \mid r.
 \end{equation}
\vskip .2cm \noindent
Let $d$ resp. $g$ be the degree resp. the genus of $C^+ \subset X^+.$ Then the following formulas are well-known (see e.g. \cite{JPR},p.603).
$$ (E^+)^3 = -rd + 2-2g; $$
$$ K_{X^+}^2 \cdot E^+ = rd + 2-2g;$$
$$ K_{X^+} \cdot (E^+)^2 = 2-2g; $$
$$ (-K_X)^3 = r^3 (L^+)^3 - 2rd + 2g-2.$$ 
Using these equations and introducing
$$ \sigma := rd + 2-2g,$$ 
we find the following relations:

\begin{enumerate}
 \item $0 = \alpha^2K_X^3 + 2\alpha\sigma + 2-2g$,
 \item $0 = \alpha({-}K_X)^3 + \beta K_F^2 - \sigma$,
 \item $0 = \alpha^2 + 2\alpha\beta K_F^2 + 2-2g$,
 \item $\left(\frac{\alpha'(\alpha'+1)(2\alpha'+1)}{12}({-}K_X)^3 +
     2\alpha'+1\right) -1 \le -\beta'(\frac{\alpha'(\alpha'+1)}{2}K_F^2 + 1)$.
\end{enumerate}
Here $K_X\cdot F^2 = 0$, $K_F^2 = K_X^2\cdot F$ and $\tilde{E}^+ =
\alpha (-K_X) + \beta F$ imply (1) and (2). Equation (3) follows from
(1) and (2). To show (4) consider the ideal sequence of $-\beta'>0$
general fibers on $X$ and twist with $\alpha'({-}K_X)$:
 \[0 \lra \sO_X(\epsilon\tilde{E}^+) \lra \sO_X({-}\alpha'K_X) \lra
 \bigoplus_{-\beta'}\sO_F(-\alpha'K_F) \lra 0,\]
with $\epsilon = 1,2$ (depending on whether $\beta' = \beta \in \KZ$
or not). Now $h^0(X, \epsilon\tilde{E}^+) = 1$ and Riemann-Roch on $X$ and $F$,
respectively, shows the claim.  

\vskip .2cm \noindent
To run a computer program we have to prove effective bounds for all
data involved. Since $-K_{X^+}|_{E^+}$ is still big and nef, we have
$\sigma > 0$. Since $\psi$ is small, $X'$ has only terminal
singularities, is hence smoothable by \cite{Namikawa}. Then the
smoothing $\X_t$ has the same index as $X$ by \cite{smoothing}, which
is $1$ by assumption. Moreover, $|{-}K_X|$ is base point free, hence
 \[4 \le ({-}K_X)^3 \le 22.\]
The image $Y^+$ of $\phi^+$ is a smooth Fano threefold of index $r$,
i.e. $1 \le r \le 4$ and
 \[0 < ({-}K_{Y^+})^3 \le \left\{\begin{array}{ll} 22, & r = 1\\ 40, & r =
     2\end{array}\right. \quad \mbox{ and }  ({-}K_{Y+})^3 = \left\{\begin{array}{ll} 54, & r = 3\\ 64, & r =
     4\end{array}\right.\]
Then $22 \ge ({-}K_X)^3 = ({-}K_Y)^3 - \sigma -rd \ge 4$ gives 
 \[d \le \frac{({-}K_{Y^+})^3 -5}{r} \le 17 \quad \mbox{ and } \quad
 \sigma \le ({-}K_{Y^+})^3-r-4 \le \left\{\begin{array}{ll} 17, & r=1\\ 34,
       & r=2\\ 47, & r=3\\ 54, & r=4\end{array}\right.\]  
Finally $\sigma = rd-2g+2$ implies $g \le \frac{17r}{2}+1$.

\vskip .2cm \noindent
Running a computer program (written in C), this leads to the following tabular:

\vspace{0.3cm}

\begin{tabular}{c|c|c|c|c|c|c|c|c}
Nr. & $r$ & $({-}K_X)^3$ & $K_F^2$ & $g$ & $d$ & $\alpha$ & $\beta$ & $(L^+)^3$\\\hline
1 & 1 & 4 & 6 & 1 & 6 & 3 & -1 & 16\\
2 & 1 & 6 & 6 & 1 & 6 & 2 & -1 & 18\\
3 & 1 & 8 & 5 & 0 & 1 & 1 & -1 & 12\\
4 & 1 & 10 & 6 & 0 & 2 & 1 & -1 & 16\\\hline
5 & 2 & 4 & 6 & 1 & 3 & 3 & -1 & 2\\
6 & 2 & 10 & 6 & 0 & 1 & 1 & -1 & 2\\
7 & 2 & 4 & 3 & 1 & 3 & 3 & -2 & 2\\
8 & 2 & 4 & 5 & 1 & 5 & 5 & -2 & 3\\
9 & 2 & 6 & 4 & 4 & 8 & 3 & -2 & 4\\
10 & 2 & 8 & 2 & 1 & 2 & 1 & -2 & 2\\
11 & 2 & 8 & 6 & 1 & 6 & 3 & -2 & 4\\
12 & 2 & 10 & 3 & 0 & 1 & 1 & -2 & 2\\
13 & 2 & 12 & 3 & 1 & 3 & 1 & -2 & 3\\
14 & 2 & 14 & 4 & 0 & 2 & 1 & -2 & 3\\
15 & 2 & 16 & 4 & 1 & 4 & 1 & -2 & 4\\
16 & 2 & 18 & 5 & 0 & 3 & 1 & -2 & 4\\
17 & 2 & 22 & 6 & 0 & 4 & 1 & -2 & 5\\\hline
18 & 3 & 4 & 5 & 9 & 11 & 8 & -3 & 2\\
19 & 3 & 14 & 4 & 5 & 8 & 2 & -3 & 2\\
20 & 3 & 16 & 5 & 3 & 7 & 2 & -3 & 2\\
21 & 3 & 18 & 6 & 1 & 6 & 2 & -3 & 2\\
22 & 3 & 8 & 8 & 8 & 10 & $\frac{7}{2}$ & $-\frac{3}{2}$ & 2\\\hline
23 & 4 & 4 & 6 & 15 & 11 & 7 & -2 & 1\\
24 & 4 & 4 & 3 & 15 & 11 & 7 & -4 & 1\\
25 & 4 & 6 & 5 & 8 & 9 & 7 & -4 & 1\\
26 & 4 & 10 & 3 & 10 & 9 & 3 & -4 & 1\\
27 & 4 & 12 & 4 & 7 & 8 & 3 & -4 & 1\\
28 & 4 & 14 & 5 & 4 & 7 & 3 & -4 & 1\\
29 & 4 & 16 & 6 & 1 & 6 & 3 & -4 & 1\\\hline
\end{tabular}

\vskip .4cm \noindent 
(3) Assume $\beta = r$. Then $r \mid (\alpha+1)$ and we several times use the following argument: let $F_1, F_2 \in
|F|$ be two general elements. Then the strict transforms $\tilde{F}_1^+, \tilde{F}_2^+
\in |\alpha L^+-\frac{\alpha+1}{r}E^+|$ cut out the exceptional curves for $\psi$ and we find the degree of the exceptional curves of $\psi$ in $Y^+$ 
 \begin{equation}\label{degpsi}
  L^+\cdot \exc(\psi) = L^+\cdot (\alpha L^+- \frac{\alpha+1}{r}E^+)^2 = \alpha^2 (L^+)^3 - (\frac{\alpha+1}{r})^2d,
\end{equation}
 hence
 \begin{equation} \label{gradarg}
   \alpha^2(L^+)^3 > (\frac{\alpha+1}{r})^2d.
 \end{equation}

\vskip .2cm \noindent We now consider the cases 1-29 seperately.
\vskip .2cm \noindent 
(3a) Assume $r=1$.

\vspace{0.2cm}

\noindent {\bf No.1,2:} open.

\vspace{0.2cm}

\noindent {\bf No.3,4:} exist by \cite{Isklines} and
\cite{Ta89}, respectively.

\vskip .2cm \noindent 
(3b) Assume $r=2$.

\vspace{0.2cm}

\noindent {\bf No.5,6:} Here $F$ would be divisible, i.e. they cannot exist.

\vspace{0.2cm}

\noindent {\bf No.7:} We have $\tilde{F}^+ = -K_{X^+} + (L^+-E^+)$, hence $h^0(X^+, L^+-E^+) = 0$. Consider the twisted ideal sequence of $C^+$ in $Y^+$
 \[0 \lra \sI_{C^+}(H^+) \lra \sO_{Y^+}(H^+) \lra \sO_{C^+}(H^+) \lra 0.\]
By Riemann-Roch on $Y^+$ we have $h^0(Y^+, H^+) = (H^+)^3+2 = 4$. Riemann-Roch on $C^+$ gives $h^0(C^+, H^+|_{C^+}) = 3$. Then $h^0(Y^+, \sI_{C^+}(H^+)) \not= 0$, a contradiction, so this case does not exist.

\vspace{0.2cm}

\noindent {\bf No.8:} Open.

\vspace{0.2cm}

\noindent {\bf No.9:} Does not exist by the same argument as in No.7.

\vspace{0.2cm}

\noindent {\bf No.10:} Excluded by \eqref{gradarg}.

\vspace{0.2cm}

\noindent {\bf No.11:} Open.

\vspace{0.2cm}

\noindent {\bf No.12:} This case exists, we will give a construction. Here
$\tilde{F}^+ = L^+-E^+$ and \eqref{degpsi} shows $\phi^+(\exc(\psi))$ is
a line. For the construction, we assume the image in $Y^+$ of a general $\tilde{F}^+_1$ is a smooth surface $S \in |H^+|$. Then the restriction of another $\tilde{F}^+_2$ to $S$ splits into $C+$ and the exceptional locus $R$ of $\psi$. We construct $R$, $C^+$, $S$ and $Y^+$ explicitely:  

Let $\nu: Y^+ \to \PN_3$ be a double covering ramified along a general quartic, i.e. $Y^+$ is a smooth Fano threefold with $-K_{Y^+} = \nu^*\sO_{\PN_3}(2) = 2H^+$. Take $S \subset Y^+$ a general element in $|H^+|$. Then $H^+$ restricts to $-K_S$ and $S$ is a smooth del Pezzo surface of degree $2$. Hence 
 \[\pi: S \to \PN_2\]
may be realized as blowup in $7$ points in general position. Let $l_1, \dots, l_7$ be the
exceptional curves for $\pi$. Define
 \[R \in |\pi^*\sO_{\PN_2}(1)-l_1-l_2|, \quad \mbox{ and } \quad C^+ \in
 |\pi^*\sO_{\PN_2}(2)-l_2-\cdots -l_7|\]
general. Then $R+C^+ = -K_S$, $R\cdot C^+ = 2$, and $R$, $C^+$ are both lines in $Y^+$. Define 
 \[X^+ = \Bl_{C^+}(Y^+).\]
   
To show $-K_{X^+}$ is nef it is enough to prove $|\sI_{C^+}(2H^+)|$ is base
point free. Consider the twisted ideal sequence 
 \[0 \lra \sI_S(2H^+) \lra \sI_{C^+}(2H^+) \lra \sI_{C^+/S}(2H^+) \lra 0.\]
Then all sections
in $H^0(S,\sI_{C^+/S}(2H^+))$ lift to $Y^+$ since $H^1(Y^+, \sI_S(2H^+)) = H^1(Y^+, \sO_{Y^+}(H^+)) = 0$. This means it suffices to prove
$|\sI_{C^+/S} \otimes \sO_{Y^+}(2H^+)|_S|$ is base point free. We have $H^+|_S = -K_S$ and 
 \[2H^+|_S-C^+ = (\pi^*\sO_{\PN_2}(2)
 -l_1-l_2-l_3-l_4) + (\pi^*\sO_{\PN_2}(2) -l_1-l_5-l_6-l_7),\]
hence the sum of two systems of quadrics in $\PN_2$ through 4 general points. These are base point free. Just numerically we find $({-}K_X)^3 = 10 > 0$ as claimed.

Since $C^+$ meets the line $R$ in $2$ points transversally, the strict transform $R^+$ of
$R$ in $X^+$ is a smooth anticanonically trivial rational curve, hence contracted by
$\psi^+$. It remains to show the flop $X$ exists and admits a del Pezzo
fibration. First the normal bundle of $R^+$ in $X^+$ is of type
$(-1,-1)$: let $S^+ \simeq S$ be the strict transform of $S$ in
$X^+$. Then $N_{R^+/S^+} =  \sO(-1)$ and we have
 \[0 \lra N_{R^+/S^+} \lra N_{R^+/X^+} \lra N_{S^+/X^+}|_{R^+} \lra 0.\]
Since $R^+$ is anticanonically trivial, the degree of $N_{R^+/X^+}$ is $2$, hence $N_{S^+/X^+}|_{R^+} = \sO(-1)$ and the sequence splits. This shows the flop is a simple flop. 

By assumption, the pencil on $X$ should be given by the strict transform of the system $|L^+-E^+|$.
The twisted ideal sequence
 \[0 \lra \sI_S(H^+) \lra \sI_{C^+}(H^+) \lra \sI_{C^+/S}(H^+) \lra 0\]
shows $h^0(Y^+, \sI_{C^+}(H^+)) = 1 + h^0(S, \sI_{C^+/S}(H^+)) = 2$ and
the base locus is exactly $R$. This gives a map $X \to \PN_1$ and
$K_F^2 = 3$ follows easily.

\vspace{0.2cm}

\noindent {\bf No.13:} Excluded by \eqref{gradarg}.

\vspace{0.2cm}

\noindent {\bf No.14:} This case exists and can be constructed the same way as No.12 above. Here $Y^+ \subset \PN_4$ is a cubic, hence a general $S \in |H^+|$ is a cubic surface, i.e. the blowup of $\PN_2$ in $6$ points. With the same notation as in No.12, define
 \[R \in |\pi^*\sO_{\PN_2}(1)-l_1-l_2| \quad \mbox{ and } \quad C^+ \in |\pi^*\sO_{\PN_2}(2)-l_3-\cdots-l_6|.\]
Then $R$ is a line and $C^+$ a conic in $Y^+$. The blowup $X^+= \Bl_{C^+}(Y^+)$ has all desired properties.

\vspace{0.2cm}

\noindent {\bf No.15:} Excluded by \eqref{gradarg}.

\vspace{0.2cm}

\noindent {\bf No.16:} Exists, the construction is as in No.12. Take the complete intersection of two quadrics in $\PN_5$ for $Y^+$. Then a general $S \in |H^+|$ is a del Pezzo surface of degree $4$, the blowup of $\PN_2$ in $5$ points. Take
 \[R \in |\pi^+\sO_{\PN_2}(1)-l_1-l_2| \quad \mbox{ and } \quad C^+ \in |\pi^*\sO_{\PN_2}(2)-l_3-\cdots-l_5|.\]
Then $R$ is a line and $C^+$ a rational curve of degree $3$ in $Y^+$ and define $X^+ = \Bl_{C^+}(Y^+)$.

\vspace{0.2cm}

\noindent {\bf No.17:} Exists, the construction is as in No.12. Take a smooth Fano threefold of type $V_{2,5}$ for $Y^+$ and $S \in |H^+|$ general. Then $S$ is the blowup of $\PN_2$ in $4$ points. Take
 \[R \in |\pi^+\sO_{\PN_2}(1)-l_1-l_2| \quad \mbox{ and } \quad C^+ \in |\pi^*\sO_{\PN_2}(2)-l_3-l_4|.\]
As above, $R$ is a line and $C^+$ a rational curve of degree $4$ in $Y^+$. Define $X^+ = \Bl_{C^+}(Y^+)$.

\vskip .2cm \noindent 
(3c) Assume $r=3$.

\vspace{0.2cm}

\noindent {\bf No.18:} Open.

\vspace{0.2cm}

\noindent {\bf No.19:} Excluded by \eqref{gradarg}.

\vspace{0.2cm}

\noindent {\bf No.20:} Exists, the construction is as No.12. Here
$\tilde{F}^+ = 2L^+-E^+$ and $\exc(\psi)$ is a line by \eqref{degpsi}. 
Let $Y^+$ be a smooth quadric and $S \subset Y^+$ a general element in $|2H^+|$. Then $S$ is a smooth del Pezzo
surface of degree $4$, hence the blowup of $\PN_2$ in $5$ points. Define
 \[R \in |\pi^*\sO_{\PN_2}(1)-l_1-l_2|, \quad \mbox{ and } \quad C^+ \in
 |\pi^*\sO_{\PN_2}(5)-l_1-l_2-2l_3-2l_4-2l_5|\]
general. Then $R+C^+ = -2K_S$, $R\cdot C^+ = 3$, $R$ is a line and $C^+$ a smooth curve of genus $3$ with $-K_S\cdot C^+ = 7$. Define $X^+ = \Bl_{C^+}(Y^+)$.

To show $-K_{X^+}$ is nef it is enough to prove $|\sI_{C^+}(3H^+)|$ is base
point free and as above it suffices to prove this for
$|\sI_{C^+/S} \otimes \sO_{Y^+}(3H^+)|_S|$, which is clear. The pencil on $X$ is defined by the strict transform of $|\sI_{C^+}(2H^+)|$, which admits exactly $2$ sections and has base locus $R$.

\vspace{0.2cm}

\noindent {\bf No.21:} Exists, the construction is similar to No.20 above. By \eqref{degpsi} the exceptional locus of $\psi$ should have degree $2$, we take the unioin of two disjoint lines $R_1$ and $R_2$ and define $C^+ = -2K_S-R_1-R_2$. More precisely: take again $S \in |2H^+|$ general and define
 \[R_1 \in |\pi^*\sO_{\PN_2}(1)-l_1-l_2|, \quad \mbox{ and } \quad R_2 \in |\pi^*\sO_{\PN_2}(1)-l_1-l_3|.\]
Then $C^+ \in |\pi^*\sO_{\PN_2}(4)-l_2-l_3-2l_4-2l_5|$ general is a smooth elliptic curve of degree $6$ intersecting each $R_i$ in $3$ points. The system $-3K_S-C^+ = -K_S+R_1+R_2$ is base point free on $S$ and $-2K_S-C^+ = R_1+R_2$ has exactly one section, hence $|2L^+-E^+|$ defines the pencil on $X$ after flopping the strict transforms of $R_1$ and $R_2$.

\vspace{0.2cm}

\noindent {\bf No.22:} This case does not exist: by \cite{Mori}, $X \subset \PN(\sF)$ for some rank $4$ vector bundle $\sF$ on $\PN_1$ and $X \in |2\zeta + \pi^*\sO(\mu)|$ for some integer $\mu$. By the same argument as in Proposition~\ref{dpdp}, II, 10.), we must have $\mu = 4$ and $c_1 = c_1(\sF) = -3$. Then
 \[\tilde{E}^+ = \frac{7}{2}({-}K_X) - \frac{3}{2}F = 7\zeta|_X + 2F.\]
Then $h^0(X, \zeta|_X) = h^0(\PN_1, \sF) = 0$, contradicting $c_1(\sF) = -3$.

\vskip .2cm \noindent 
(3d) $Y^+$ has index $4$, i.e. $Y^+ = \bP_3.$

\vspace{0.2cm}

\noindent {\bf No.23:} Cannot exist since $F$ is not divisible.

\vspace{0.2cm}

\noindent {\bf No.24:} Suppose that $C^+ $ lies on a cubic:
$$ H^0(\sI_{C^+}(3)) \ne 0.$$ 
Therefore 
$$h^0(3L^+-E^+) \ne 0 \eqno (*) $$ 
Using $7(-K_{X^+}) = E^+ + 4\tilde F$, putting in $L^+$ and dividing by $4$ we obtain
$7 L^+ = 2E^+ + \tilde F$, hence
$$ H^0(7L^+ - 2 E^+) = 2.$$ 
But $ 7L^+ - 2 E^+ = 2(3L^+ - E^+) + L^+$ which gives via (*) an inequality $h^0(7L^+-2E^+) \geq 4.$   
Hence $C^+$ does not lie on a cubic. Since $C^+$ is contained in a quadric (since $h^0(-K_{X^+}) \ne 0$),  we may apply a theorem of
Gruson-Peskine, see \cite{Ha87},p.151, and obtain $g(C^+) \leq 15,$ a contradiction.

\vspace{0.2cm}

\noindent {\bf No.25:} Open. Here the argument of No.24 does not work. 

\vspace{0.2cm}

\noindent {\bf No.26:} Here $3L^+ - E^+ = \tilde F,$ hence
$$ h^0(\sI_{C^+}(3)) = 2.$$ 
By reasons of degree, $C^+$ is the intersection of two cubics $Q_i = \phi_+(\tilde F_i).$ 
But $Q_1 \cap Q_2$ must contain the images of curves which are contracted by $\psi^+,$ hence must contain rational curves. 
This rules out No.26.

\vspace{0.2cm}

\noindent {\bf No.27:} Exists, the construction is the same as No.12. By \eqref{degpsi} the degree of the exceptional curves of $\psi$ is $1$. Take $S \in |3H^+|$ general. Then $S$
is a smooth cubic, hence the blowup of $\PN_2$ in $6$ points. Take
  \[R \in |\pi^*\sO_{\PN_2}(1)-l_1-l_2|, \quad C^+ \in
 |\pi^*\sO_{\PN_2}(8)-2l_1-2l_2-3l_3-\cdots -3l_6|\]
and define $X^+ = \Bl_{C^+}(Y^+)$.

\vspace{0.2cm}

\noindent {\bf No.28:} Exists, the construction is as in No.21. By \eqref{degpsi} the degree of the exceptional curves is $2$, we take again two disjoint lines $R_1$ and $R_2$ in $S$, i.e.
 \[R_1 \in |\pi^*\sO_{\PN_2}(1)-l_1-l_2|, \quad R_2 \in |\pi^*\sO_{\PN_2}(1)-l_1-l_3|.\]
Then $C^+ \in |\pi^*\sO_{\PN_2}(7)-2l_2-2l_3-3l_4-3l_5-3l_6|$ is a smooth curve of degree $7$ and genus $4$ intrsecting each $R_i$ in $4$ points. The system $|{-}4K_S-C^+|$ is base point free and $|{-}3K_S-C^+|$ onedimensional. Define $X^+=\Bl_{C^+}(\PN_3)$ as usually.

\vspace{0.2cm}

\noindent {\bf No.29:} Exists, the construction is as above: by \eqref{degpsi} the exceptional locus of $\psi$ should have degree $3$, hence take $3$ disjoint lines $R_1, R_2$ and $R_3$ in the cubic $S \in |3H^+|$:
 \[R_1 \in |\pi^*\sO_{\PN_2}(1)-l_1-l_2|, \quad R_2 \in |\pi^*\sO_{\PN_2}(1)-l_1-l_3|, \quad R_3 \in
 |\pi^*\sO_{\PN_2}(1)-l_2-l_3|.\]
Then $C^+ \in |{-}3K_S-R_1-R_2-R_3|$ general is a smooth elliptic curve of degree $6$ intersecting each $R_i$ in $4$ points. The blowup $X^+ = \Bl_{C^+}(Y^+)$ has all desired properties.
\qed     


\section{Conic bundles} \label{secconic}
\setcounter{lemma}{0}

\begin{setup} {\rm In this section $\phi: X \to Y = \bP_2$ denotes a conic bundle with $\rho(X) = 2$. 
As always we assume $-K_X$ big and nef but not ample
and that the anticanonical morphism is small; moreover $-K_X$ is spanned.
The discriminant locus is denoted by $\Delta$. Set 
$$\tau = \deg \Delta.$$
We introduce the rank 3-bundle
$$ \sE = \phi_*(-K_X).$$ 
By \cite{JPR} $\sE$ is spanned, since $\psi$ is not divisorial
(compare the proof of Proposition~3.2 in \cite{JPR}). Thus we obtain an embedding 
$$ X \subset  \bP(\sE)$$
such that $-K_X = \zeta \vert X.$ 
The divisor $X \subset \bP(\sE)$ is of the form
$$ [X] = 2 \zeta + \pi^*(\sO(\lambda)) $$
with some integer $\lambda.$ Then the adjunction formula yields
$$ \lambda = 3 - c_1.$$  
Here we use the shorthand $c_i = c_i(\sE).$ 
Since
$$ H^q(\bP(\sE),- \zeta - \pi^*(\sO(\lambda))) = 0$$
for $q = 0,1,$ every section in $H^0(-K_X)$ uniquely lifts to a section of $\zeta.$ 
Thus $\vert \zeta \vert $ defines via Stein factorisation a  map $\hat \psi: \bP(\sE) \to \bP'  $ extending $\psi$  and in total
a map $\sigma \circ \hat \psi: \bP(\sE) \to \hat W \subset \bP_{g+1}.$ \\
Now we consider the flopping diagram \ref{flopdiag}. 
{\it Since the case $\dim Y^+ = 1 $ is already settled by sect. 3, we will always assume that 
$$ \dim Y^+ \geq 2,$$
so that $Y^+= \bP_2$ (and $\phi^+$ is a conic bundle) or $\dim Y^+ = 3.$} \\
As usual, we let $L = \phi^*\sO_{\PN_1}(1)$ and $L^+ $ be the pull-back
to $X^+$ of the ample generator $H^+$
on $Y^+$. The ``strict transform'' of $L$ in $X^+$ is denoted $\tilde L $ and 
similarly the strict transform of $L^+$ in $X$ is denoted $\tilde L^+.$ \\
In the case $\dim Y^+ = 3$ we denote the excptional divisor of $\phi^+$ by
$E^+$ and its strict transform in $X$ by $\tilde{E}^+$. If $E^+$ contracts to a smooth curve $C^+$, we let $g = g(C^+)$
and $d = H^+\cdot C^+$ the genus and degree of $C^+$. The index of $Y^+$ will be $r$ if $Y^+$ is Gorenstein; the only non-Gorenstein case occurs when $E^+ = \bP_2$ with normal bundle $\sO(-2).$
Then $-2K_{Y^+} $ is Cartier and we define $r$ by
$$ (\phi^+)^*(-2K_{Y^+}) = rL^+.$$ 

}
\end{setup}

\begin{proposition} \label{p1}
If $\phi$ is a $\PN_1$-bundle, then $r_X = 2$.
\end{proposition}

\begin{proof}
We write $X = \PN(\sF)$ with $\eta$ the tautological line bundle and normalize $\sF$ such that $c_1(\sF) = 0, -1$. If $c_1(\sF) = -1$, then $-K_X = 2\eta + 4L$, i.e. $r_X = 2$. We may hence assume $c_1(\sF) = 0$.

Consider a curve $l_{\psi}$ contracted by $\psi$ and let $C = \phi(l_{\psi})$. We may assume that $C$ is smooth (otherwise normalize). Write
 \[\sF|_C = \sO_{\PN_1}(a) \oplus \sO_{\PN_1}(-a)\]
and set $e = 2a$. Since $C_0 := l_{\psi}$ is contractible in $\PN(\sF_C)$, we have $C_0^2 = -e$. Since $-K_X = 2\eta +3L$ and since $\eta|_{\PN(\sF_C)} = C_0+af$ where $f=l_{\varphi}$ is a ruling line, we obtain
 \[0 = -K_X \cdot l_{\psi} = (2C_0+(2a+3d)f)\cdot C_0 = -e+3d,\]
where $d = \deg(C)$ is the degree of $C$.

On the other hand, $\phi_*(-K_X) = S^2(\sF(1))$ not globally generated implies $\psi$ divisorial by \cite{JPR}, Proposition~3.2. (compare the proof of 3.2., in particular p. 588. There we do not assume $\psi$ divisorial but show it.). Hence $\sF(1)$ is nef, which gives $a \le 1$. Then $a=1$ and $e=2=3d$, which is impossible.
\end{proof}

\begin{proposition}  Assume $\Delta\not= \emptyset$ and write
$$ \tilde L = \alpha^+ (-K_{X^+}) + \beta^+ L^+  \eqno (*)$$
and 
$$ \tilde L^+ = \alpha (-K_X) + \beta L . \eqno (**)$$
with $\alpha^+,\beta^+, \alpha, \beta \in \bQ.$  \\
Then ${\rm Pic}(X) = \bZ (-K_X) + \bZ L,$ hence 
$\alpha, \beta \in \bZ,$ and one of the following cases occur.
$$\alpha = 2\alpha^+ \  {\rm and} \  \beta^+ = {{-1} \over {2}}, \beta = -2; \leqno (1)$$ 
\item $$ \alpha^+ = \alpha \  {\rm and} \  \beta^+ = \beta = -1. \leqno (2)$$ 
\end{proposition}

\proof First note that intersecting with an irreducible component of a
reducible conic gives $\alpha, \beta \in \bZ$, and intersectung with
an extremal rational curve of $\phi^+$ gives $2\alpha^+, 2\beta^+ \in
\bZ$. Moreover, $\alpha^+, \alpha \ge 0$. Putting now equation (*) into (**) and having in mind $- \tilde K_X = -K_{\tilde X}$ yields
$$  \alpha + \beta \alpha^+ = 0$$
and
$$ \beta^+ \beta = 1.$$
By symmetry we also have
$$ \alpha^+ + \beta^+ \alpha = 0.$$
Now a trivial calculation gives (1) and (2).
\qed

\vskip .4cm
Of course (*) can be rewritten as
$$ L = \alpha^+ (-K_X) + \beta^+ \tilde L^+$$
and analogously for (**).
\vskip .2cm \noindent
We shall also consider a general fiber $l_{\phi}$ and also a general fiber $l_{\phi^+}$ of $\phi+$ if $\dim Y^+ = 2.$ 
If $\phi^+$ is birational, we let $l_{\phi^+}$ be a minimal rational curve contracted by $\phi^+.$ 
The intersection numbers with $-K_X$ resp. $-K_{X^+}$ are either $1$ or $2$;
for $l_{\phi}$ the number is $2.$ The general $l_{\phi}$ will not meet the exceptional
locus of $\psi;$ thus it lies naturally in $X^+,$ and we denote the completed family in $X^+$ by $l^+_{\phi}.$
The same for $l_{\phi^+}$ if $\dim Y^+ = 2; $ here the notation is
$\tilde l_{\phi^+}.$ 

\

We start with the case that $\phi^+: X^+ \to Y^+ =
\bP_2$ is a conic bundle.

\begin{theorem} \label{conicconic}
If $\phi$ is a proper conic bundle, then the case $\dim Y^+ = 2$ is impossible.
\end{theorem}

\begin{proof}
Intersect  
$$ \alpha^+ (-K_X) = L + \tilde L^+ \eqno (*)$$  
with a conic $l_{\phi} $ to obtain
$$ \tilde L^+ \cdot l_{\phi} = 2 \alpha^+.$$ 
Hence
$$ L^2 \cdot \tilde L^+ = 2 \alpha^+.$$ 
Analogously
$$ \tilde L \cdot (L^+)^2 = 2 \alpha^+. $$
Cube equation (*), so that 
$$ (\tilde L^+)^3 = (\alpha^+)^3 (-K_X)^3 - 12 \alpha^+. \eqno (**)$$ 
Now observe that  
$$(\tilde L^+)^3 < 0.$$
This can be seen either by a spectral sequence argument 
plus Riemann-Roch, computing $\chi(\tilde L^+),$ or as follows. Take two general elements  
$$S_i \in \vert \tilde L^+ \vert. $$ 
Then 
$$ S_1 \cdot S_2 = \sum a_i l_i$$
where $a_i > 0$ and the $l_i$ are contracted by $\psi.$ Since $\tilde L^+ \cdot l_i < 0,$ 
we obtain $( \tilde L^+)^3 < 0.$ \\ 
Thus (**) yields
$$ (\alpha^+)^2 (-K_X)^3 < 12, \eqno (***)$$
hence either $\alpha^+ = 1$ and $({-}K_X)^3 \le 11$ or $\alpha^+ = 2$ and
$({-}K_X)^3 = 2$. 

\vspace{0.2cm}

\noindent Assume $\alpha^+ = \alpha = 1$. Then $-K_X = L + \tilde L^+$ and
hence 
 \[-K_X\cdot L^2 = ({-}K_X)^2\cdot L - ({-}K_X)\cdot L\cdot \tilde L^+
 = ({-}K_X)^2\cdot L - L^2\cdot \tilde L^+ - L \cdot (\tilde L^+)^2.\]
We obtain $2 = (12-\tau) -4$, hence $\tau = 6$. Analogously we prove
$\tau^+ = 6$, with $\tau^+$ the degree of the discriminant locus of $\phi^+$. Then $K_X^2\cdot (\tilde L^+)^2 = ({-}K_X)^3 - K_X^2 \cdot L$ yields
$({-}K_X)^3 = 12$, contradicting (***). 

\noindent If $\alpha^+ = \alpha = 2$, the analogous computation gives $\tau =
\tau^+ = 9$ and  $({-}K_X)^3 = 3$, which is again impossible.
\end{proof}

\

{\bf From now on we assume $\tau\not= 0$ and $\phi^+: X^+ \to Y^+$ is birational.}

\begin{lemma}
$\phi^+$ cannot be the blow-up of a smooth point. 
\end{lemma} 

\begin{proof} Assume that $\phi^+$ is the blow-up of the smooth point $p$ and let $E^+$ be the exceptional divisor. 
Clearly $E^+$ cannot contain any curve $l_{\psi^+}$ so that the general line $l^+ := l_{\phi^+} \subset E^+$ does not meet the exceptional
set of $\psi^+.$ Let $l' = \psi^+(l^+).$ Let $\X' \to \Delta$ be a
smoothing of $X'$. Then
$$ N_{l'/X'} = \sO(-1) \oplus \sO(1)$$ and
$$ N_{l'/{\mathcal X'}} = \sO(-1) \oplus \sO(1) \oplus \sO.$$ 
Hence $l'$ moves to the smooth fibers $X'_t.$ Let $l'_t \subset X'_t$ be such a deformation. Then 
$$-K_{X'_t} \cdot l'_t = 2$$ so that
$l'_t$ is a conic in the smooth Fano threefold $X'_t.$ Thus the deformations of $l'_t$ inside $X_t'$ fill up $X'_t$ ($t \ne 0$). 
But then the deformations of $l'$ in $X'$ must fill up $X'$, which is absurd. 
\end{proof}

\begin{proposition} The case $\beta = -2$ is impossible. Moreover $\alpha^+ = \alpha \in \bN.$ 
\end{proposition}

\begin{proof}  
Suppose $\beta = -2$ so that $\beta^+ = -{{1} \over {2}}.$ 
First we claim that there cannot be a curve $C$  contracted by $\phi^+$ such that $-K_{X^+} \cdot C = 1.$ In fact, if such a curve exists,
then
$$ \tilde L \cdot C = \alpha^+ - {{1} \over {2}} L^+ \cdot C = \alpha^+,$$
hence $\alpha^+$ is an integer and therefore $L^+$ is divisible by $2$ which is absurd. 
Thus $\phi^+$ cannot be a proper conic bundle.
It cannot be a $\bP_1-$bundle either by assumption. So - recalling that we assume $\dim Y^+ \ne 1,$ the morphism
$\phi^+$ is birational and by Mori's classification the non-existence of a curve $C$ with $-K_{X^+} \cdot C = 1$  forces $\phi^+$
to be the blow-up of a smooth point in $\tilde Y$ which is excluded by the last lemma. \\
Thus $\beta = -1$ and therefore also $\beta^+ =  -1.$ Consider the decomposition $\alpha^+ (-K_X) = L + \tilde L^+ $ and intersect with the 
irreducible component $l$ of a reducibe conic: $\alpha^+ = L \cdot l + \tilde L^+ \cdot l \in \bN. $    
\end{proof} 

\begin{corollary} Suppose $\phi^+$ birational. Then
$\beta^+ = \beta = -1$ and $\alpha^+ = \alpha \in \bN.$ Moreover 
$$ \tilde{E}^+ = (r\alpha^+ - 1) (-K_X) - rL $$
unless $E^+ = \bP_2$ with normal bundle $\sO(-2).$ In that case
$$ \tilde{E}^+ = (r\alpha^+ - 2) (-K_X)  - rL.$$ 
\end{corollary} 

\begin{lemma} Suppose $\phi^+$ birational. If $E^+ \ne \bP_2, $ then 
$$ ({{(r\alpha^+-1)^3} \over {6}} + {{(r\alpha^+-1)^2} \over {4}} + {{(r\alpha^+-1)} \over {12}})(-K_X)^3 
+(\tau - 12)({{(r\alpha^+-1)^2r} \over {2}}+ {{(r\alpha^+-1)r} \over {2}}) + $$
$$ + (r\alpha^+-1)(r^2+2) + {{r^2} \over {2}} - {{3} \over {2}}r + 1 \leq 1.$$
If $E^+ = \bP_2,$ then
$$ ({{(r\alpha^+-2)^3} \over {6}} + {{(r\alpha^+-2)^2} \over {4}} + {{(r\alpha^+-2)} \over {12}})(-K_X)^3 
+(\tau - 12)({{(r\alpha^+-2)^2r} \over {2}}+ {{(r\alpha^+-2)r} \over {2}}) + $$
$$ + (r\alpha^+-2)(r^2+2) + {{r^2} \over {2}} - {{3} \over {2}}r + 1 \leq 1.$$
\end{lemma}

\begin{proof} The left hand side of the inequality is just $\chi(\sO_X(\tilde{E}^+))$ (Riemann-Roch and the last corollary). 
Thus it remains to show that $\chi(\sO_X(\tilde{E}^+)) \leq 1.$ This follows from 
$$ H^q(\sO_X(\tilde{E}^+)) = 0$$
for $q \geq 2$ which is an easy application of the Leray spectral sequence and the obvious vanishing
$$ H^q(\sO_{X^+}(E^+)) = 0$$
for $q \geq 2.$ 
\end{proof}

\begin{lemma} \label{absch2}
Suppose $\phi^+$ birational.
\begin{enumerate} 
\item If $E^+ = Q_2,$ then
$$ (r(12-\tau)-2)(r\alpha^+ -1) = 2r^2 + 2. $$
\item If $E^+ = \bP_2,$ then
$$  (r(12-\tau)-1)(r\alpha^+ -2) = 2r^2 + 2. $$
\item If $E^+$ is ruled, then 
$$ (r\alpha^+-1)(rd+2-2g-r(12-\tau)) = (2g-2) - 2r^2.$$
\end{enumerate} 
\end{lemma} 

\begin{proof} Let us say that we are in case (1) or (3). 
Then we compute $E^3$ in two ways; putting both equations together gives our
claim. The first equation is
$$ ((r\alpha^+-1)(-K_X) - E)^3 = 0,$$
the second 
$$E^3 = ((r\alpha^+-1)(-K_X) - rL)^3.$$
\end{proof} 

\begin{lemma} \label{absch3}
Suppose $\phi^+$ birational. 
\begin{enumerate} 
\item If $E^+ = Q_2,$ then
$$ (r\alpha^+-1)^3 (-K_X)^3 - 6(r\alpha^+-1)^2 - 6(r\alpha^+-1)  - 2 < 0.$$
\item If $E^+ = \bP_2,$ then
$$ (r\alpha^+-2)^3 (-K_X)^3 - 3(r\alpha^+-2)^2 - 6(r\alpha^+-2)  - 4 < 0.$$
\item If $E^+$ is ruled, then 
$$ (r\alpha^+-1)^3 (-K_X)^3 - 3(r\alpha^+-1)^2(rd+2-2g) + 3(r\alpha^+-1)(2g-2)  +rd + 2g-2\leq 0.$$
\end{enumerate} 
\end{lemma}

\begin{proof} Notice that $ \tilde L^3 < 0.$ 
Then we equate 
$$ r\tilde L^3 = ((r\alpha^+-1)(-K_{X^+}) -  E^+)^3.$$
\end{proof}

\begin{proposition} \label{conicpt} 
Suppose $E^+$ contracts to a point. 
\begin{enumerate}
\item If $E^+ = Q_2$, then $({-}K_{X^+})^3 = 8$, $r=1$, $(L^+)^3 =
  10$, $\tau = 6$ and $\alpha^+ = 2$.
\item If $E^+ = \PN_2$, then $({-}K_{X^+})^3 = 6$, $r=1$, $(L^+)^3 =
  52$, $\tau = 7$ and $\alpha^+ = 3$.
\end{enumerate}
Both cases really exist.
\end{proposition}

\begin{proof}
First note that $Y^+$ is singular at the image $p$ of $E^+$. Since $r
= 4$ implies $Y^+ \simeq \PN_3$, we have $r \le 3$. If $r = 3$, then
$Y^+$ must be the quadric cone with vertex $p$. But a divisorial
resolution of the quadric cone does not have $\rho = 2$. So this case
is again impossible. We end up with $r \le 2$.

1.) Assume $E^+ = Q_2$ is a quadric. Set $a:= r\alpha^+-1$ for
short. Then Lemma~\ref{absch3} gives $a(a^2({-}K_{X^+})^3 -6a-6) \le 1$. But $a^2({-}K_{X^+})^3 -6a-6 > 0$ leads to $({-}K_{X^+})^3 = 13$ which
is impossible. Hence $a^2({-}K_{X^+})^3 -6a-6\le 0$. We find $a \le 2$
using $({-}K_{X^+})^3\ge 4$ since $X'$ is not hyperelliptic by
Corollary~\ref{hypsymm}.

We have $({-}K_{X^+})^3 = r^3(L^+)^3 -2$ and $Y^+$ has a terminal
Gorenstein singularity at $p$. This means $Y^+$ is smoothable, and the
pair $(r, (L^+)^3)$ must correspond to a smooth Fano
threefold of Picard number one. Using now the numerical conditions
above a short computer program gives the case as stated in the proposition is the only solution.

\vspace{0.2cm}

2.) Assume $E^+ = \PN_2$. Then $Y^+$ has a $2$--Gorenstein terminal
singularity at $p$ and $8({-}K_{X^+})^3 = r^3(L^+)^3 -4$. Setting $a =
r\alpha^+ -2$ Lemma~\ref{absch3} together with $({-}K_{X^+})^3 \ge 4$
gives $a = 1$ and hence $r = 1$. Then $\tau = 7$ by Lemma~\ref{absch2}
and finally $a({-}K_X)^3 = (E+rL)\cdot K_X^2 = 1 + r(12 - \tau)$
implies $({-}K_{X^+})^3 = 6$.  

\vspace{0.2cm}

It remains to show the existence.

(1) (i) We start constructing $X' \subset \PN_6$. Let $x_0, \dots, x_6$ be
homogeneous coordinates of $\PN_6$, $l_0, \dots, l_5$ general linear
forms and $Q$ a general quadric. Then the complete intersection of
 \[Q_0 = x_0l_0+x_1l_1+x_2l_2, \quad Q_1 = x_0l_3+x_1l_4+x_2l_5 \quad \mbox{and} \quad Q\]
is a Fano threefold $X'$ of index $r = 1$ and degree
$({-}K_{X'})^3 = 8$ containing the quadric surface $E' \subset \PN_3$ defined
by $Q$ and $x_0=x_1=x_2=0$. A computer calculation (Macaulay) shows $X'_{sing}
\subset E'$ consists of $6$ points (you can also check this
directly). Note that the quadrics $Q_0$ and $Q_1$ have exactly one
singular point not contained in $Q$.

\vspace{0.2cm}

(ii) Now construct the first small resolution $X$. Blowing up the
$\PN_3$ containing $E'$ in $\PN_6$ resolves the rational map to $\PN_2$
defined by $x_0, x_1, x_2$. We obtain 
 \[\xymatrix{\PN(\sO^{\oplus 4} \oplus \sO(1))
   \ar[d]^{\phi} \ar[r]^{\hspace{1cm}\psi}&
   \PN_6\\
          \PN_2 &}\] 
with exceptional divisor $D$ and tautological line bundle
$\zeta$. Denote $F = \phi^*\sO_{\PN_2}(1)$. Then $D \in |\zeta
-F|$. The strict transforms $\hat{Q_0}, \hat{Q_1} \in |\zeta+F|$ and
$\hat{Q} = \psi^*Q \in |2\zeta|$ cut out the smooth almost Fano threefold $X$ with $-K_X
= \zeta|_X$.

The complete intersection $\hat{Q_0} \cap \hat{Q_1}$ is a
$\PN_2$-bundle $\PN(\sE) \to \PN_2$ defined by an exacct sequence
 \[0 \lra \sO^{\oplus 2}(-1) \lra \sO^{\oplus 4} \oplus \sO(1) \lra \sE \lra 0\]
on $\PN_2$. This is a small resolution of the complete intersection $Q_0 \cap Q_1$, which
is a Fano fourfold of index $3$. The singular locus of $Q_0 \cap Q_1$
is the rational normal curve $C$ in $\PN_3$; the exceptional divisor of
$\PN(\sE) \to Q_0 \cap Q_1$ is a $\PN_1$-bundle over $C$. This can be
seen as follows: $\hat{Q_0}$ restricted to $D =
\PN(N^*_{\PN_3/\PN_6})$ is a section in
$|\zeta+\psi^*\sO_{\PN_3}(2)|$, hence a section of
$\sO_{\PN_3}(1)^{\oplus 3}$. This gives a $\PN_1$-bundle
over $\PN_3$ outside the only singular point of $Q_0$; here the fiber
of $D \cap \hat{Q_0}$ is a $\PN_2$.  

The intersection with $\hat{Q_1}$, a further section in
$\sO_{\PN_3}(1)^{\oplus 3}$ gives generically an isomorphism with
exceptional locus exactly the points, where the map
 \[\sO_{\PN_3}^{\oplus 2} \lra \sO_{\PN_3}(1)^{\oplus 3}\]
has not rank two. This defines the rational normal curve.  

\vspace{0.2cm}

The induced map $\phi: X \to \PN_2$ is then a
proper conic bundle defined by $\hat{Q}|_{\PN(\sE)} \in |2\zeta|$,
where again $\zeta$ is the tautological line bundle on $\PN(\sE)$. The discrininant $\Delta$ is the vanishing
locus of the determinant of the map
 \[\sE^* \lra \sE\]
induced by the section $\hat{Q}$ of $|2\zeta|$. We find $\tau =
\deg(\Delta) = 6$. The intersection of $Q$ with the rational normal
curve $C$ gives $6$ points, the singularities of $X'$. The exceptional
locus of $\psi$ consists of $6$ single $\PN_1$'s over these $6$ points.

\vspace{0.2cm}

(iii) Finally construct the flop $X^+$. First define the strict
transform $\tilde{E}^+$ of the quadric surface $E' \simeq \PN_1 \times \PN_1$ in
$\PN_6$. Then $\tilde{E}^+$ is the blowup of $E'$ is $6$ points, hence a del
Pezzo surface of degree $2$.

We claim $X^+ \to Y^+$ is birational, contracting the strict transform
$E^+ \simeq \PN_1 \times \PN_1$ of $\tilde{E}^+$ to a singular point. Note $-K_X|_{\tilde{E}^+} = \psi^*\sO_{\PN_2}(1)$, $\psi|_{\tilde{E}^+}$
being the blowdown of the six $({-}1)$-curves to $E' \simeq \PN_1
\times \PN_1$.

For the normal bundle of the $\psi$-exceptional
curves $C_1, \dots, C_6$ we find $N_{C_i/\tilde{E}^+} = \sO_{\PN_1}(-1)$, hence
$N_{C_i/X}$ is of type $(-1,-1)$ and $X^+$ may be obtained as simple flop
 \[\xymatrix{ & Z = \Bl_{C_1, \dots, C_6}(X) \ar[dl]_p \ar[dr]^q &\\
              X \ar@{<-->}[rr] && X^+}\]
Let $\hat{E}$ be the strict transform of $\tilde{E}^+$ in $Z$. Then $\hat{E}
\simeq \tilde{E}^+$ and 
 \[K_Z|_{\hat{E}} = K_X|_{\tilde{E}^+} + \sum C_i = \psi^*\sO_{\PN_1 \times \PN_1}(-1,-1) + \sum C_i.\]
Blowing down the exceptional divisors of $p$ the other direction, we
obtain $E^+ = q(\hat{E}) \simeq \PN_1 \times \PN_1$, where $q|_{\hat{E}} =
\psi|_{\tilde{E}^+}$. Let $K_{X^+}|_{E^+} = \sO_{\PN_1\times\PN_1}(a,b)$. Then using $q$, we find
 \[K_Z|_{\hat{E}} = q^*K_{X^+}|_{E^+} + \sum C_i =
 q^*\sO_{\PN_1\times \PN_1}(a,b) + \sum C_i.\]
This shows $a=b = -1$, and hence $N_{E^+/X^+} = \sO(-1)$ by
adjunction. The map contracting $E^+$ to a point is defined by
$|{-}2K_{X^+}-F^+|$, where $F^+$ is the strict transform of $F$. We find
$|{-}2K_{X^+}-F^+|$ is nef and trivial on $E^+$.

\vspace{0.5cm}

(2) (i) We start constructing $X' \subset \PN_5$. Let $x_0, \dots, x_5$ be
homogeneous coordinates of $\PN_5$, $l_0, l_1, l_2$ general linear
forms and $q_0, q_1, q_2$ three general quadrics. Then the complete intersection of
 \[Q = x_0l_0+x_1l_1+x_2l_2, \quad \mbox{and} \quad K =
 x_0q_0+x_1q_1+x_2q_2\]
is a Fano threefold $X'$ of index $r = 1$ and degree
$({-}K_{X'})^3 = 6$ containing the surface $E' \simeq \PN_2$ defined
by $x_0=x_1=x_2=0$. A computer calculation (Macaulay) shows $X'_{sing}
\subset E'$ consists of $7$ points (you can also check this directly). 

\vspace{0.2cm}

(ii) Now construct the first small resolution $X$. Blowing up $E' \subset \PN_5$ resolves the rational map to $\PN_2$
defined $x_0, x_1, x_2$. We obtain 
 \[\xymatrix{\PN(\sO^{\oplus 3} \oplus \sO(1))
   \ar[d]^{\phi} \ar[r]^{\hspace{1cm}\psi}&
   \PN_5\\
          \PN_2 &}\] 
with exceptional divisor $D$ and tautological line bundle
$\zeta$. Denote $F = \phi^*\sO_{\PN_2}(1)$. Then $D \in |\zeta
-F|$. The strict transforms $\hat{Q} \in |\zeta+F|$ and $\hat{K} \in
|2\zeta + F|$ cut out the smooth almost Fano threefold $X$ with $-K_X
= \zeta|_X$.

The $\PN_2$-bundle $\hat{Q} \to \PN_2$ is defined by the quotient
 \[\sO_{\PN_2}^{\oplus 3} \oplus \sO_{\PN_2}(1) \lra \sE = T_{\PN_2}(-1)
   \oplus \sO_{\PN_2}(1) \lra 0,\]
with exceptional divisor $D_Q = \hat{Q} \cap D$ a $\PN_1$-bundle over
$E'$. The induced map $\phi: X \to \PN_2$ is then a proper conic bundle defined by $\hat{K}|_{\hat{Q}} \in
|2\zeta + F|$, where again $\zeta$ is the tautological line bundle on
$\hat{Q} = \PN(\sE)$. The discrininant $\Delta$ is the vanishing
locus of the determinant of the map
 \[\sE^* \lra \sE \otimes \sO_{\PN_2}(1)\]
induced by the section $\hat{K}$ of $|2\zeta + F|$. We find $\tau =
\deg(\Delta) = 7$.

The restriction of $\hat{K}$ to $D_Q$ is a del
Pezzo surface $\tilde{E}^+$ of degree $2$, the blowup of $E' = \PN_2$ in the seven
singular points of $X'$. Hence $\psi|_X$ is small, contracting seven
rational curves $C_1, \dots, C_7$ to points.

\vspace{0.2cm}

(iii) Finally construct the flop $X^+$. We claim $X^+ \to Y^+$ is
birational, contracting the strict transform $E^+ \simeq \PN_2$ of $\tilde{E}^+$
to a singular point. Note $-K_X|_{\tilde{E}^+} = \psi^*\sO_{\PN_2}(1)$, $\psi|_{\tilde{E}^+}$
being the blowdown of the seven $({-}1)$-curves..

For the normal bundle of the $\psi$-exceptional
curves $C_1, \dots, C_7$ we find $N_{C_i/\tilde{E}^+} = \sO(-1)$, hence
$N_{C_i/X}$ is of type $(-1,-1)$ and $X^+$ may be obtained as simple flop
 \[\xymatrix{ & Z = \Bl_{C_1, \dots, C_7}(X) \ar[dl]_p \ar[dr]^q &\\
              X \ar@{<-->}[rr] && X^+}\]
Let $\hat{E}$ be the strict transform of $\tilde{E}^+$ in $Z$. Then $\hat{E}
\simeq \tilde{E}^+$ and 
 \[K_Z|_{\hat{E}} = K_X|_{\tilde{E}^+} + \sum C_i = \psi^*\sO_{\PN_2}(-1) + \sum C_i.\]
Blowing down the exceptional divisors of $p$ the other direction, we
obtain $E^+ = q(\hat{E}) \simeq \PN_2$, where $q|_{\hat{E}} =
\psi|_{\tilde{E}^+}$. Let $K_{X^+}|_{E^+} = \sO_{\PN_2}(\lambda)$. Then using $q$, we find
 \[K_Z|_{\hat{E}} = q^*K_{X^+}|_{E^+} + \sum C_i =
 q^*\sO_{\PN_2}(\lambda) + \sum C_i.\]
This shows $\lambda = -1$, and hence $N_{E^+/X^+} = \sO(-2)$ by
adjunction. The map contracting $E^+$ to a point is defined by
$|{-}3K_{X^+}-F^+|$, where $F^+$ is the strict transform of $F$. We find
$|{-}3K_{X^+}-F^+|$ is nef and trivial on $E^+$.
\end{proof}

\begin{lemma} Suppose $E^+$ contracts to a curve. Then $$(r \alpha^+ -
  1)(-K_X)^3 = d + 2-2g + r(12 - \tau)$$ and $\alpha^+ \le 8$. 
\end{lemma}

\begin{proof} The first claim follows from $(r\alpha^+ - 1)(-K_X) = \tilde{E}^+ + rL $ together with 
$ d + 2-2g  = K_{X^+}^2 \cdot E^+ = K_X^2 \cdot \tilde{E}^+$ and $K_X^2 \cdot L = 12 - \tau.$

\noindent Using $12-\tau \le 11$ and $rd + 2-2g = r^3(L^+)^3-rd
-({-}K_{X^+})^3$ in the above equation gives 
 \[\alpha^+ \le \frac{r^2(L^+)^3 + 10}{({-}K_{X^+})^3} \le \frac{r^2(L^+)^3 + 10}{4}.\] 
To see the last inequality note that $X'$ cannot by hyperelliptic by
Corollary~\ref{hypsymm}. Hence $({-}K_{X^+})^3 \ge 4$. Now the bounds
for $r$ and $(L^+)^3$ from Iskovskikh's list prove $\alpha^+ \le 8$.
\end{proof} 

Putting things together numerical computer calculations (written in C) show

\begin{proposition} \label{complist}
$X$ and $X^+$ have the following invariants. 

\begin{center} \begin{tabular}{c|c|c|c|c|c|c|c|}
 No. & $({-}K_X)^3$ & $\alpha^+$ & $r$ & $(L^+)^3$ & $d$ & $g$ &
 $\tau$\\\hline  
 1 & $4$ & $3$ & $1$ & $8$ & $1$ & $0$ & $7$\\
 2 & $4$ & $4$ & $1$ & $18$ & $8$ & $2$ & $6$\\
 3 & $4$ & $5$ & $1$ & $22$ & $10$ & $2$ & $4$\\
 4 & $4$ & $3$ & $2$ & $4$ & $9$ & $5$ & $7$\\
 5 & $4$ & $4$ & $2$ & $5$ & $11$ & $5$ & $5$\\
 6 & $6$ & $3$ & $1$ & $14$ & $3$ & $0$ & $5$\\
 7 & $6$ & $2$ & $2$ & $3$ & $5$ & $2$ & $7$\\
 8 & $6$ & $3$ & $2$ & $4$ & $6$ & $0$ & $4$\\
 9 & $6$ & $2$ & $3$ & $2$ & $11$ & $10$ & $7$\\
 10 & $8$ & $2$ & $3$ & $2$ & $9$ & $5$ & $5$\\
 11 & $10$ & $2$ & $1$ & $14$ & $1$ & $0$ & $5$\\
 12 & $10$ & $2$ & $2$ & $5$ & $8$ & $2$ & $4$\\
 13 & $10$ & $2$ & $3$ & $2$ & $7$ & $0$ & $3$\\
 14 & $12$ & $2$ & $1$ & $18$ & $2$ & $0$ & $4$\\
 15 & $14$ & $2$ & $1$ & $22$ & $3$ & $0$ & $3$\\
 16 & $18$ & $1$ & $4$ & $1$ & $6$ & $2$ & $4$\\
 17 & $22$ & $1$ & $3$ & $2$ & $5$ & $0$ & $3$\\\hline
\end{tabular}
\end{center}
\end{proposition} 

\begin{theorem} \label{coniccurve}
If $\phi^+$ is divisorial contracting $E^+$ to a curve, then $X$ is one of 1,3,5,6,8,10,11,12,13,14,15,16,17 in Proposition~\ref{complist}.
\end{theorem} 

\begin{proof} We will go case by case through the list in Proposition~\ref{complist}.

\vspace{0.3cm}

\noindent {\bf No.1:} This is classical and can be found in \cite{AG5}, so the existence is clear.

\vspace{0.3cm}

\noindent {\bf No.2:} Here we compute 
$$ \chi(\sO_X(\tilde{E}^+)) = 1.$$ 
By Leray spectral sequence arguments we have
$$ H^2(\sO_X(\tilde{E}^+)) = H^2(\sO_{X'}(E')) = H^2(\sO_{X^+}(E^+)) = 0.$$ 
Since $H^3(\sO_X(\tilde{E}^+)) = 0$ anyway, it follows
$$ H^1(\sO_X(\tilde{E}^+)) = 0.$$ 
Again the Leray spectral sequence gives 
$$ R^1\psi_*(\sO_X(\tilde{E}^+)) = 0$$
so that $E^+ \cdot l = 1$ for all curves contracted by $\psi^+.$ Hence $\psi^+ \vert {E^+}$ is
biholomorphic and thus $E' = \psi^+(E^+) $ is a smooth surface in $X' \subset \bP_4$ (recall $(-K_X)^3 = 4$). 
Moreover we know that $E' \subset \bP_4$ has degree $6.$ This contradicts e.g. the double point formula. 

\vspace{0.3cm}

\noindent {\bf No.3:} Open. 

\vspace{0.3cm}

\noindent {\bf No.4:} This is parallel to Case 7 below: again
$C^+$ is degenerate, now in $\bP_5$ and the contradiction is the same.  

\vspace{0.3cm}

\noindent {\bf No.5:} Open.

\vspace{0.3cm}

\noindent {\bf No.6:} This is classical as no.1.

\vspace{0.3cm}

\noindent {\bf No.7:} Here $Y^+ \subset \bP_4$ is a cubic. Castelnuovo's bound \cite{ACGH85},p.116 implies that $C^+ \subset \bP_4$
is degenerate, so that 
$$ H^0(\sI_{C^+}(1)) \ne 0.$$
Consequently 
$$ H^0(-K_X-E) \ne 0.$$ 
On the other hand, $3(-K_X) = E + 2L,$ hence 
$$ h^0(-3K_X - E) = 6.$$ 
This is obviously a contradiction, since $h^0(-2K_X) > 6.$ 

\vspace{0.3cm}

\noindent {\bf No.8:} Open.

\vspace{0.3cm}

\noindent {\bf No.9:} Here $Y^+ = Q_3$ is a quadric. We have $-5K_{X^+} = E^+ + 3\tilde L$ and
$-K_{X^+} = 3L^+ - E^+$. This gives $\tilde L = -K_{X^+} + (2L^+ - E^+)$
on $X^+$. Then $h^0(-K_{X^+}) = 6$ but $h^0(X^+, \tilde L) = 3$
implies 
 \[H^0(X^+, 2L^+ - E^+) = 0.\] 
Now the ideal sequence of $C^+$ in $Y^+ = Q_3$ gives 
 \[0 \lra H^0(Q_3, \sO(2)) \lra H^0(C^+, \sO(2)|_{C^+})\] 
is injective, hence $h^0(C^+, \sO(2)) \ge 14$. On the other hand, $h^0(C^+, \sO_{Q_3}(2)|_C) = 13$ by Riemann--Roch.

\vspace{0.3cm}

\noindent {\bf No.10:} Open.

\vspace{0.3cm}

\noindent {\bf No.11:} Classical as no.1.

\vspace{0.3cm}

\noindent {\bf No.12:} Open.

\vspace{0.3cm}

\noindent {\bf No.13:} Open.

\vspace{0.3cm}

\noindent {\bf No.14:} Classical as no.1.

\vspace{0.3cm}

\noindent {\bf No.15:} Classical as no.1.

\vspace{0.3cm}

\noindent {\bf No.16:} We will give a construction. Let $Y^+ = \PN_3$ and $S \in
|\sO_{\PN_3}(3)|$ a smooth cubic. Write $\pi\colon S \to \PN_2$ the
blowup of $6$ general points in $\PN_2$ and denote the exceptional
curves of $\pi$ by $l_1, \dots, l_6$. Choose
 \[C \in |\pi^*\sO(4) - 2l_1-l_2- \cdots -l_5|\] 
general. Then $C \subset \PN_3$ is a smooth curve of degree
$6$ and genus $2$. Define $X^+ = \Bl_C(\PN_3)$ with exceptional
divisor $E^+$. We want to see $X^+$ is almost Fano.

Let $L^+$ be the pullback of $\sO_{\PN_3}(1)$ and
denote the strict transform of $S$ in $|3L^+ - E^+|$ on $X^+$ again by
$S$. Then $-K_{X^+} = 4L^+-E^+$ and we obtain $({-}K_{X^+})^3 =
18$. Since the restriction map $H^0(X^+, -K_{X^+}) \lra H^0(S,
-K_{X^+}|_S)$ is surjective, it suffices to show that $-K_{X^+}|_S$ is
base point free. We have  
 \[-K_{X^+}|_S = \sO_{\PN_3}(4) \otimes \sO_S(-C) = \pi^*\sO(8)
 -2l_1-3l_2 - \cdots - 3l_5 -4l_6.\]
Intersecting with the $27$ lines on the cubic surface $S$ shows the
claim.

It remains to show that the anticanonical map $\psi^+$ of $X^+$ is
small and that the flop $X$ is a conic bundle with
discriminant locus of degree $4$. Consider
 \[\tilde L := 3L^+-E^+.\] 
Then $|\tilde L|$ is base point free outside $S \in |\tilde L|$, and the restriction map 
 \[0 \lra H^0(X^+, \sO_{X^+}) \lra H^0(X^+, \tilde L) \lra H^0(S, \tilde L|_S) \lra 0\] 
is surjective with onedimensional kernel. On $S$, we have 
 \[\tilde L|_S = \pi^*\sO(5) - l_1 - 2l_2 - \cdots - 2l_5 -3l_6.\] 
Subtracting the unique section $l^+_{\psi} \in |\pi^*\sO(2) - \sum_{i=2}^6 l_i|$ we get $l^+_{\phi} =
\tilde L|_S - l^+_{\psi}$ is a pencil. This shows the base locus of
$\tilde L$ is the single rational curve $l^+_{\psi}$ and $h^0(X^+,
\tilde L) = 3$. 

Flopping $l^+_{\psi}$ we therefore obtain a morphism onto $\PN_2$,
which then must be a conic bundle. The degree of the discriminant
locus may now be computed easily.

\vspace{0.3cm}

\noindent {\bf No.17:} We will give a construction analogously to the last case. Let $Y^+ = Q_3$ be a smooth
quadric and $S \in |\sO_{Q_3}(2)|$ be a general member. Then $S$ is a
smooth del Pezzo surface of degree $4$, hence $\pi\colon S \to
\PN_2$ is the blowup of $5$ general points. Denote the exceptional
curves of $\pi$ by $l_1, \dots, l_5$ as above. Choose
 \[C \in |\pi^*\sO(2) - l_1|\] 
general. Then $C \subset Q_3$ is a smooth rational curve of degree
$5$. Define $X^+ = \Bl_C(Q_3)$ with exceptional divisor $E^+$. Denote
the strict transform of $S$ again by $S$ and the pullback of
$\sO_{Q_3}(1)$ by $L^+$. Then
$({-}K_{X^+})^3 = 22$ and
 \[-K_{X^+}|_S = \sO_{Q_3}(3)|_S \otimes \sO_S(-C) = \pi^*\sO(7) -2l_1-3l_2 - \cdots - 3l_5.\]
This implies $-K_{X^+}$ nef since the points are general.

Finally consider $\tilde L := 2L^+-E^+$. Then 
 \[\tilde L|_S = \pi^*\sO(4) - l_1 - 2l_2 - \cdots - 2l_5\] 
and substracting the unique section $l^+_{\psi} \in |\pi^*sO(2) - \sum l_i|$ we get $l^+_{\phi} =
\tilde L|_S - l^+_{\psi}$ is a pencil. This shows the base locus of
$\tilde L$ is $l^+_{\psi}$ and $h^0(X^+, \tilde L) = 3$. This
completes the construction as in the last case.
\end{proof} 


\begin{appendix}

\section{Tables}

Some general notation: let $-K_X = r_X H$ and $-K_{X^+} = r_{X^+}H^+$. For the vector bundle $\bigoplus_{i=1}^m\sO_{\PN_1}(a_i)^{\oplus d_i}$ on $\PN_1$ write $(a_1^{d_1}, \dots, a_m^{d_m})$ for short. A ``+'' in the last column indicates that the case exists, a ``?'' means the existence is unknown.

\subsection{Del Pezzo fibration -- Del Pezzo fibration} \label{tdpdp}

Assume $X$ and $X^+$ admit del Pezzo fibrations with general fiber $F$ and $F^+$, respectively. Let $\tilde{F}^+$ be the strict transform of $F^+$ under the flopping map. Let $\sE = \phi_*H$, $\sE^+ = \phi^+_*H^+$  and assume
 \[\tilde{F}^+  = \alpha H +\beta F.\]
We find in any case that $\sE$ and $\sE^+$ are of the same type. Let $\lambda$ be the maximal integer such that
 \[H^0(X, H-\lambda F) \not= 0.\]

\begin{center}
\renewcommand{\arraystretch}{1.3}
\begin{tabular}{c|c|c|c|c|c|c|c|c|c|c|}
No. & $({-}K_X)^3$ & $r_X$ & $K_F^2$ & $K_{F^+}^2$ & $\sE/\sE^+$ & $\alpha$ & $\beta$ & $\lambda$ & Ref. & $\exists$\\\hline
$1$ & $54$ & $3$ & $9$ & $9$ & $(0,1^2)$ & $1$ & $-1$ & $1$ & \ref{highindex} & +\\\hline
$2$ & $32$ & $2$ & $8$ & $8$ & $(0^2,1^2)$ & $1$ & $-1$ & $1$ & \ref{index2}, (1.iii) & +\\\hline
$3$ & $16$ & $2$ & $8$ & $8$ & $(0^4)$ & $2$ & $-1$ & $0$ & \ref{index2}, (1.i) & +\\\hline
$4$ & $16$ & $1$ & $8$ & $4$ & $(0, 1^2, 2^2)$ & $\frac{1}{2}$ & $-\frac{1}{2}$ & $2$ & \ref{beta} & +\\\hline
$5$ & $12$ & $1$ & $6$ & $6$ & $(0^5, 1^2)$ & $1$ & $-1$ & $1$ & \ref{dpdp}, I & ?\\\hline
$6$ & $10$ & $1$ & $5$ & $5$ & $(0^4, 1^2)$ & $1$ & $-1$ & $1$ & \ref{dpdp}, I & ?\\\hline
$7$ & $8$ & $1$ & $4$ & $4$ & $(0^3, 1^2)$ & $1$ & $-1$ & $1$ & \ref{dpdp}, I & +\\\hline
$8$ & $6$ & $1$ & $6$ & $6$ & $({-}1, 0^6)$ & $2$ & $-1$ & $0$ & \ref{dpdp}, II & ?\\\hline
$9$ & $6$ & $1$ & $3$ & $3$ & $(0^2, 1^2)$ & $1$ & $-1$ & $1$ & \ref{dpdp}, I & +\\\hline 
$10$ & $4$ & $1$ & $6$ & $6$ & $({-}1^2, 0^5)$ & $3$ & $-1$ & $0$ & \ref{dpdp}, II & ?\\\hline
$11$ & $4$ & $1$ & $4$ & $4$ & $(0^5)$ & $2$ & $-1$ & $0$ & \ref{dpdp}, II & +\\\hline
$12$ & $4$ & $1$ & $2$ & $2$ & $(0, 1^2)$ & $1$ & $-1$ & $1$ & \ref{dpdp}, I & +\\\hline
$13$ & $2$ & $1$ & $6$ & $6$ & $({-}1^3, 0^4)$ & $6$ & $-1$ & $0$ & \ref{dpdp}, II & ?\\\hline
$14$ & $2$ & $1$ & $5$ & $5$ & $({-}1^2, 0^4)$ & $5$ & $-1$ & $0$ & \ref{dpdp}, II & ?\\\hline
$15$ & $2$ & $1$ & $4$ & $4$ & $({-}1, 0^4)$ & $4$ & $-1$ & $0$ & \ref{dpdp}, II & +\\\hline
$16$ & $2$ & $1$ & $3$ & $3$ & $(0^4)$ & $3$ & $-1$ & $0$ & \ref{dpdp}, II & +\\\hline 
$17$ & $2$ & $1$ & $1$ & $1$ & $(1^2)$ & $1$ & $-1$ & $1$ & \ref{basepoints} & +\\\hline
\end{tabular}
\end{center}

\

\subsection{Del Pezzo fibration -- Conic bundle} \label{tdelpezzoconic}

Assume $X \to \PN_1$ admits a del Pezzo fibration with general fiber $F$ and $X^+ \to \PN_2$ is a conic bundle with discriminant locus of degree $\tau$. Let $L^+$ be the pullback of $\sO_{\PN_2}(1)$ and $\tilde{L}^+$ the strict transforms under the flopping map. Then we find in any case
 \[\tilde{L}^+  = H -F.\]
Let $\sE = \phi_*H$ and $\sF^+ = \phi^+_*H^+$ with Chern classes $c_i=c_i(\sF^+)$.

\begin{center}
\renewcommand{\arraystretch}{1.3}
\begin{tabular}{c|c|c|c|c|c|c|c|c|}
No. & $({-}K_X)^3$ & $r_X$ & $K_F^2$ & $\sE$ & $\tau$ & $(c_1,c_2)$ & Ref. & $\exists$\\\hline
$1$ & $40$ & $2$ & $8$ & $(0, 1^3)$ & $0$ & $(3,4)$ & \ref{index2}, (1.iv) & +\\\hline
$2$ & $14$ & $1$ & $6$ & $(0^4, 1^3)$ & $4$ & $(5,13)$ & \ref{delpezzoconic} & ?\\\hline
$3$ & $12$ & $1$ & $5$ & $(0^3, 1^3)$ & $5$ & $(4,8)$ & \ref{delpezzoconic} & ?\\\hline
$4$ & $10$ & $1$ & $4$ & $(0^2, 1^3)$ & $6$ & $(3,4)$ & \ref{delpezzoconic} & +\\\hline
$5$ & $8$ & $1$ & $3$ & $(0, 1^3)$ & $7$ & $(2,1)$ & \ref{delpezzoconic} & +\\\hline
\end{tabular}
\end{center}

\

\subsection{Del Pezzo fibration -- Birational contraction with $\dim(\phi^+(E^+)) = 0$} \label{tdppoint}

Assume $X \to \PN_1$ admits a del Pezzo fibration with general fiber $F$ and $X^+ \to Y^+$ is a birational contraction with exceptional divisor $E^+$ contracted to a point. Then $Y^+$ is a (possibly singular) Fano threefold with $\rho(Y^+) = 1$. Let $L^+$ be the pullback of the generator of $\Pic(Y^+)$ and $\tilde{L}^+$ the strict transforms under the flopping map. Then we find in any case
 \[\tilde{L}^+  = H -F.\]
Let $\sE = \phi_*H$.

\begin{center}
\renewcommand{\arraystretch}{1.3}
\begin{tabular}{c|c|c|c|c|c|c|c|c|}
No. & $({-}K_X)^3$ & $r_X$ & $K_F^2$ & $\sE$ & $({-}K_{Y^+})^3$ & $(E^+, E^+|_{E^+})$ & Ref. & $\exists$\\\hline
$1$ & $24$ & $2$ & $8$ & $(0^3,1)$ & $32$ & $(\PN_2, \sO({-}1))$ & \ref{index2}, (1.ii) & +\\\hline
$2$ & $10$ & $1$ & $6$ & $(0^6,1)$ & $18$ & $(\PN_2, \sO({-}1))$ & \ref{blowupsmooth} & +\\\hline
$3$ & $6$ & $1$ & $4$ & $(0^4,1)$ & $8$ & $(Q, \sO({-}1))$ & \ref{blowupq} & +\\\hline
$4$ & $4$ & $1$ & $3$ & $(0^3,1)$ & $\frac{9}{2}$ & $(\PN_2, \sO({-}2))$ & \ref{blowupsing} & +\\\hline
\end{tabular}
\end{center}

\

\subsection{Del Pezzo fibration -- Birational contraction with $\dim(\phi^+(E^+)) = 1$} \label{tdpcurve}

Assume $X \to \PN_1$ admits a del Pezzo fibration with general fiber $F$ and $X^+ \to Y^+$ is a birational contraction with exceptional divisor $E^+$ contracted to a smooth curve of degree $d$ and genus $g$. Then $r_X = 1$ and $Y^+$ is a smooth Fano threefold of index $r_{Y^+}$ with $\rho(Y^+) = 1$. Let $\tilde{E}^+$ be the strict transform of $E^+$ under the flopping map. Assume
 \[\tilde{E}^+  = \alpha H + \beta F.\]
The number in the column ``Ref.'' refers to the table in Proposition~\ref{dpcurve}.

\begin{center}
\renewcommand{\arraystretch}{1.3}
\begin{tabular}{c|c|c|c|c|c|c|c|c|c|c|}
No. & $({-}K_X)^3$ & $K_F^2$ & $r_{Y^+}$ & $({-}K_{Y^+})^3$ & $g$ & $d$ & $\alpha$ & $\beta$ & Ref. & $\exists$\\\hline
$1$ & $22$ & $6$ & $2$ & $40$ & $0$ & $4$ & $1$ & $-2$ & \ref{dpcurve} (17) & +\\\hline
$2$ & $18$ & $5$ & $2$ & $32$ & $0$ & $3$ & $1$ & $-2$ & \ref{dpcurve} (16) & +\\\hline
$3$ & $18$ & $6$ & $3$ & $54$ & $1$ & $6$ & $2$ & $-3$ & \ref{dpcurve} (21) & +\\\hline
$4$ & $16$ & $6$ & $4$ & $64$ & $1$ & $6$ & $3$ & $-4$ & \ref{dpcurve} (29) & +\\\hline
$5$ & $16$ & $5$ & $3$ & $54$ & $3$ & $7$ & $2$ & $-3$ & \ref{dpcurve} (20) & +\\\hline
$6$ & $14$ & $5$ & $4$ & $64$ & $4$ & $7$ & $3$ & $-4$ & \ref{dpcurve} (28) & +\\\hline
$7$ & $14$ & $4$ & $2$ & $24$ & $0$ & $2$ & $1$ & $-2$ & \ref{dpcurve} (14) & +\\\hline
$8$ & $12$ & $4$ & $4$ & $64$ & $7$ & $8$ & $3$ & $-4$ & \ref{dpcurve} (27) & +\\\hline
$9$ & $10$ & $6$ & $1$ & $16$ & $0$ & $2$ & $1$ & $-1$ & \ref{dpcurve} (4) & +\\\hline
$10$ & $10$ & $3$ & $2$ & $16$ & $0$ & $1$ & $1$ & $-2$ & \ref{dpcurve} (12) & +\\\hline
$11$ & $8$ & $6$ & $2$ & $32$ & $1$ & $6$ & $3$ & $-2$ & \ref{dpcurve} (11) & ?\\\hline
$12$ & $8$ & $5$ & $1$ & $12$ & $0$ & $1$ & $1$ & $-1$ & \ref{dpcurve} (3) & +\\\hline
$13$ & $6$ & $6$ & $1$ & $18$ & $1$ & $6$ & $2$ & $-1$ & \ref{dpcurve} (2) & ?\\\hline
$14$ & $6$ & $5$ & $4$ & $64$ & $8$ & $9$ & $7$ & $-4$ & \ref{dpcurve}
(25) & ?\\\hline
$15$ & $4$ & $6$ & $1$ & $16$ & $1$ & $6$ & $3$ & $-1$ & \ref{dpcurve} (1) & ?\\\hline
$16$ & $4$ & $5$ & $2$ & $24$ & $1$ & $5$ & $5$ & $-2$ & \ref{dpcurve} (8) & ?\\\hline
$17$ & $4$ & $5$ & $3$ & $54$ & $9$ & $11$ & $8$ & $-3$ & \ref{dpcurve} (18) & ?\\\hline
\end{tabular}
\end{center}

\

\subsection{Conic bundle -- Conic bundle} \label{tcbcb}

Assume $X \to \PN_2$ and $X^+ \to \PN_2$ both are conic bundles. Let $\tau$ and $\tau^+$ be the degree of their discriminant loci. By Proposition~\ref{conicconic} then $\tau = \tau^+ = 0$ and $r_X = 2$. This case was treated in \cite{delpezzo}; we obtain two completely symmetric cases: denote $\sF = \phi_*H$ with Chern classes $c_i = c_i(\sF)$ and $\sF^+ = \phi^+_*H^+$ with Chern classes $c_i^+=c_i(\sF^+)$. Then $c_i = c_i^+$.

\begin{center}
\renewcommand{\arraystretch}{1.3}
\begin{tabular}{c|c|c|c|c|c|c|}
No. & $({-}K_X)^3$ & $r_X$ & $\tau$ & $(c_1, c_2)$ & Ref. & $\exists$\\\hline
$1$ & $24$ & $2$ & $0$ & $(3,6)$ & \ref{index2}, (2.iii) & +\\\hline
$2$ & $16$ & $2$ & $0$ & $(3,7)$ & \ref{index2}, (2.iv) & +\\\hline
\end{tabular}
\end{center}

\

\subsection{Conic bundle -- Birational contraction with $\dim(\phi^+(E^+)) = 0$} \label{tconicpt}

Assume $X \to \PN_2$ is a conic bundle, $\tau$ the degree of the discriminant locus. Denote $\sF = \phi_*H$ with Chern classes $c_i = c_i(\sF)$ and $L = \phi^*\sO_{\PN_2}(1)$. Assume $X^+ \to Y^+$ is a birational contraction with exceptional divisor $E^+$ contracted to a point. Then $Y^+$ is a (possibly singular) Fano threefold with $\rho(Y^+) = 1$. Let $L^+$ be the pullback of the generator of $\Pic(Y^+)$ and $\tilde{L}^+$ the strict transform under the flopping map. Assume
 \[\tilde{L}^+  = \alpha H + \beta L.\]

\begin{center}
\renewcommand{\arraystretch}{1.3}
\begin{tabular}{c|c|c|c|c|c|c|c|c|c|c|}
No. & $({-}K_X)^3$ & $r_X$ & $\tau$ & $(c_1, c_2)$ & $({-}K_{Y^+})^3$ & $(E^+, E^+|_{E^+})$ & $\alpha$ & $\beta$ & Ref. & $\exists$\\\hline
$1$ & $32$ & $2$ & $0$ & $(3,5)$ & $40$ & $(\PN_2, \sO(-1))$ & $2$ & $-1$ & \ref{index2}, (2.ii) & +\\\hline
$2$ & $8$ & $1$ & $6$ & $(1,0)$ & $10$ & $(Q, \sO(-1))$ & $2$ & $-1$ & \ref{conicpt} & +\\\hline
$3$ & $6$ & $1$ & $7$ & $(1,0)$ & $52$ & $(\PN_2, \sO(-2))$ & $3$ & $-1$ & \ref{conicpt} & +\\\hline
\end{tabular}
\end{center}

\

\subsection{Conic bundle -- Birational contraction with $\dim(\phi^+(E^+)) = 1$} \label{tconicurve}

Assume $X \to \PN_2$ is a conic bundle, $\tau$ the degree of the discriminant locus. Denote $L = \phi^*\sO_{\PN_2}(1)$. Assume $X^+ \to Y^+$ is a birational contraction with exceptional divisor $E^+$ contracted to a smooth curve of degree $d$ and genus $g$. Then $r_X = 1$ and $Y^+$ is a smooth Fano threefold of index $r_{Y^+}$ with $\rho(Y^+) = 1$. Let $\tilde{L}^+$ be the strict transform of $L^+$, the pullback of the generator of $\Pic(Y^+)$, under the flopping map. Then
 \[\tilde{L}^+  = \alpha H  - L.\]
The number in the column ``Ref.'' refers to the table in Proposition~\ref{coniccurve}.

\begin{center}
\renewcommand{\arraystretch}{1.3}
\begin{tabular}{c|c|c|c|c|c|c|c|c|c|}
No. & $({-}K_X)^3$ & $\tau$ & $({-}K_{Y^+})^3$ & $r_{Y^+}$ & $d$ & $g$ & $\alpha$ & Ref. & $\exists$\\\hline
$1$ & $22$ & $3$ & $54$ & $3$ & $5$ & $0$ & $1$ & \ref{coniccurve} (17) & +\\\hline
$2$ & $18$ & $4$ & $64$ & $4$ & $6$ & $2$ & $1$ & \ref{coniccurve} (16) & +\\\hline
$3$ & $14$ & $3$ & $22$ & $1$ & $3$ & $0$ & $2$ & \ref{coniccurve} (15) & +\\\hline
$4$ & $12$ & $4$ & $18$ & $1$ & $2$ & $0$ & $2$ & \ref{coniccurve} (14) & +\\\hline
$5$ & $10$ & $5$ & $14$ & $1$ & $1$ & $0$ & $2$ & \ref{coniccurve} (11) & +\\\hline
$6$ & $10$ & $4$ & $40$ & $2$ & $8$ & $2$ & $2$ & \ref{coniccurve} (12) & ?\\\hline
$7$ & $10$ & $3$ & $24$ & $2$ & $7$ & $0$ & $2$ & \ref{coniccurve} (13) & ?\\\hline
$8$ & $8$ & $5$ & $54$ & $3$ & $9$ & $5$ & $2$ & \ref{coniccurve} (10) & ?\\\hline
$9$ & $6$ & $5$ & $14$ & $1$ & $3$ & $0$ & $3$ & \ref{coniccurve} (6) & +\\\hline
$10$ & $6$ & $4$ & $32$ & $2$ & $6$ & $0$ & $3$ & \ref{coniccurve} (8) & ?\\\hline
$11$ & $4$ & $7$ & $8$ & $1$ & $1$ & $0$ & $3$ & \ref{coniccurve} (1) & +\\\hline
$12$ & $4$ & $5$ & $40$ & $2$ & $11$ & $5$ & $4$ & \ref{coniccurve} (5) & ?\\\hline
$13$ & $4$ & $4$ & $22$ & $1$ & $10$ & $2$ & $5$ & \ref{coniccurve} (3) & ?\\\hline
\end{tabular}
\end{center}

\

\end{appendix}



\begin{thebibliography}{JPR05}
\bibitem[ACGH85]{ACGH85} E. Arbarello, M. Cornalba, P.A. Griffiths, J. Harris: Geometry of algebraic curves. Springer 1985. 
\bibitem[BS95]{BS95} M. Beltrametti, A.J. Sommese: The adjunction theory of complex projective varieties. de Gruyter 1995.
\bibitem[Be07]{Bertini} E. Bertini: Introduzione alla geometria
  proiettiva degli iperspazi. E. Spoerri, Pisa, 1907.
\bibitem[Ch99]{Ch99} I. Cheltsov: Three-Dimensional Fano Varieties of
    Integer Index. Math. Notes {\bf 66}, 360-365 (1999).
\bibitem[CSP05]{Cheltsov} I. Cheltsov, C. Shramov, V. Przyjalkowski:
  Hyperelliptic and trigonal Fano threefolds. Izv. Math. {\bf 69}, No.2, 365-421 (2005).
\bibitem[Fu90]{Fu90} T. Fujita: Classification theories of polarized varieties. London Math. Soc. Lect. Notes Ser. {\bf 155} (1990).
\bibitem[Ha87]{Ha87} R. Hartshorne: On the classification of algebraic space curves. II. Proc. Symp. Pure Math. {\bf 46}, No.1, 145-164 (1987). 
\bibitem[Isk78]{Isk} V.A. Iskovskikh: Fano $3$-folds I, II. Math. USSR, Izv. {\bf 11}, 485-527 (1977); {\bf 12}, 469-506 (1978).
\bibitem[Isk80]{Iskantican} V.A. Iskovskikh: Anticanonical models of three-dimensional algebraic varieties. J. Soviet Math. {\bf 13}, 745-814 (1980).
\bibitem[Isk89]{Isklines} V.A. Iskovskikh: Double projection from a line
  onto Fano threefolds of the first kind. (Eng. Transl) Math. USSR
  Sb. {\bf 66}, 265-284 (1990). 
\bibitem[IP99]{AG5} V.A. Iskovskikh, Yu.G. Prokhorov: Algebraic
  Geometry V: Fano varieties. Springer 1999.
\bibitem[JP06]{delpezzo} P. Jahnke, T. Peternell: Almost del Pezzo
  manifolds. math.AG/0612516.
\bibitem[JPR05]{JPR} P. Jahnke, T. Peternell, I. Radloff: Threefolds
  with big and nef anticanonical bundles I. Math. Ann. {\bf 333}, No.3, 569-631 (2005).
\bibitem[JR06a]{smoothing} P. Jahnke, I. Radloff: Terminal Fano
    threefolds and their smoothings. math.AG/0610769.
\bibitem[JR06b]{JR} P. Jahnke, I. Radloff: Gorenstein Fano threefolds
  with base points in the anticanonical system. Comp. Math. {\bf 142},
  No.2, 422-432 (2006).
\bibitem[Ko89]{Kollar} J. Koll\'ar: Flops. Nagoya Math. J. {\bf 113},
  15-36 (1989).
\bibitem[Mo82]{Mori} S. Mori: Threefolds whose canonical bundles are
  not numerically effective. Ann. Math. {\bf 116}, 133-176 (1982).
\bibitem[Na97]{Namikawa} Y. Namikawa: Smoothing Fano 3-folds. J. Algebr. Geom. {\bf 6}, 307-324 (1997). 
\bibitem[Shi89]{Shin} Shin, K.-H.: $3$-dimensional Fano Varieties with Canonical Singularities. Tokyo J. Math. {\bf 12}, 375-385 (1989).
\bibitem[Ta89]{Ta89} K. Takeuchi: Some birational maps of Fano
  $3$-folds. Comp. Math. {\bf 71}, 265-284 (1989).
\end{thebibliography}
\end{document}